\documentclass[11pt,a4paper]{amsart}
\usepackage{amsmath,amssymb,amscd,amsthm}

\newlength{\mysidemargin}
\newlength{\mytopmargin}
\begin{document}


\newcommand{\hogehoge}[1]{ \bigskip \noindent ********** \\ {#1} \\ ********** \bigskip }


\newtheorem{theo}{ {\bf Theorem} }
\newtheorem{lemm}[theo]{ {\bf Lemma} }
\newtheorem{prop}[theo]{ {\bf Proposition} }
\newtheorem{cor}[theo]{ {\bf Corollary} }
\newtheorem{asse}{ {\bf Assertion}}    

\newtheorem{defn}{{\bf  Definition}}    
\newtheorem{rem}{{\bf  Remark}}    

\renewcommand{\proofname}{{\rm{Proof.}}} 


\newcommand{\NN}{{\mathbb N}}
\newcommand{\ZZ}{{\mathbb Z}}
\newcommand{\QQ}{{\mathbb Q}}
\newcommand{\RR}{{\mathbb R}}
\newcommand{\CC}{{\mathbb C}}
\newcommand{\HH}{{\mathbb H}}


\newcommand{\im}{\mathop{\rm{Im}}\nolimits}

\newcommand{\lcm}{\mathop{\rm{lcm}}\nolimits}

\newcommand{\etp}{{\mathbf{e}}}

\newcommand{\hol}{\mathop{\rm{Hol}}\nolimits}

\newcommand{\diag}{\mathop{\rm{diag}}\nolimits}   


\newcommand{\SL}{{\mathrm S \mathrm L}}
\newcommand{\GL}{{\mathrm G \mathrm L}}
\newcommand{\SP}{{\mathrm S \mathrm p}}
\newcommand{\MT}{{\mathrm M}}


\newcommand{\ABCD}{\begin{pmatrix} A & B \\ C & D \end{pmatrix}}
\newcommand{\OEEOG}{\begin{pmatrix} O_g & -E_g \\ E_g & O_g \end{pmatrix}}
\newcommand{\OEEO}{\begin{pmatrix} O_2 & -E_2 \\ E_2 & O_2 \end{pmatrix}}

\newcommand{\ELMa}{\begin{pmatrix} -1&0&0&0\\0&1&0&0\\0&0&-1&0\\0&0&0&1 \end{pmatrix}}
\newcommand{\ELMb}{\begin{pmatrix} 0&1&0&0\\1&0&0&0\\0&0&0&1\\0&0&1&0 \end{pmatrix}}
\newcommand{\ELMc}{\begin{pmatrix} 1&0&0&0\\x&1&0&0\\0&0&1&-x\\0&0&0&1 \end{pmatrix}}
\newcommand{\ELMd}{\begin{pmatrix} 1&x&0&0\\0&1&0&0\\0&0&1&0\\0&0&-x&1 \end{pmatrix}}

\newcommand{\ELMe}{\begin{pmatrix} a&0&b&0\\0&1&0&0\\c&0&d&0\\0&0&0&1 \end{pmatrix}}
\newcommand{\ELMf}{\begin{pmatrix} 1&0&0&0\\0&a&0&b\\0&0&1&0\\0&c&0&d \end{pmatrix}}

\newcommand{\abcd}{\begin{pmatrix} a & b \\ c & d \end{pmatrix}}

\newcommand{\tzzo}{\begin{pmatrix} \tau & z \\ z & \omega \end{pmatrix}}
\newcommand{\tnno}{\begin{pmatrix} \tau & 0 \\ 0 & \omega \end{pmatrix}}


\newcommand{\Gcenter}[2]{\dimen0=\ht\strutbox\advance\dimen0\dp\strutbox\multiply\dimen0
by#1\divide\dimen0 by2\advance\dimen0 by-.5\normalbaselineskip\raisebox{-\dimen0}[0pt][0pt]{#2}}



\newcommand{\M}{{\mathrm M}}
\newcommand{\LCM}{{\mathrm L\mathrm C \mathrm M}}

\def\C{{\mathbb C}}
\def\FF{{\mathbb F}}
\def\HH{{\mathbb H}}
\def\N{{\mathbb N}}
\def\P{{\mathbb P}}
\def\Q{{\mathbb Q}}
\def\R{{\mathbb R}}
\def\Z{{\mathbb Z}}
\def\E{{\mathcal E}}


\newcommand{\Sym}{{\mathrm S \mathrm y \mathrm m}}

\newcommand{\frct}{\frac{a \tau + b}{c \tau +d}}
\newcommand{\frco}{\frac{a \omega + b}{c \omega +d}}

\newcommand{\LES}[3]{0 & \to & {#1} & \to & {#2} &
                     \stackrel{ P_1 (r) }{ \longrightarrow } & {#3}}
\newcommand{\LEs}[3]{0 & \!\!\to\!\! & {#1} & \!\!\to\!\! & {#2} &
                     \!\!\stackrel{ P_2 (s) }{ \longrightarrow }\!\! & {#3}}
\newcommand{\Les}[3]{0 & \to & {#1} & \to & {#2} &
                     \stackrel{ P_3 (s)}{ \longrightarrow } & {#3}}
\newcommand{\TXT}[2]{( \text{{#1}} {#2} )}
\newcommand{\CAE}[1]{({#1} \text{ even})}
\newcommand{\CAO}[1]{({#1} \text{ odd})}
\newcommand{\CSE}{{\CAE{s}}}
\newcommand{\CSO}{{\CAO{s}}}
\newcommand{\PDL}[2]{{\frac{\partial {#1}}{\partial {#2}}}}
\newcommand{\BBACK}{\!\!\!\!\!\!}


\newcommand{\MVA}{{\mathit A}}
\newcommand{\MVB}{{\mathrm W}}
\newcommand{\MVC}{{\mathit A}}

\newcommand{\MVAa}[1]{\MVA_{#1}}
\newcommand{\MVBa}[1]{\MVB_{#1}}
\newcommand{\MVCa}[1]{\MVC_{#1}}

\newcommand{\MVAG}{\MVAa{*} (\Gamma)}
\newcommand{\MVAGI}{\MVA_*^{\natural} (\Gamma)}
\newcommand{\MVAGa}[1]{\MVA_{#1} (\Gamma)}
\newcommand{\MVAGb}[2]{\MVA_{#1} (\Gamma ; {#2})}
\newcommand{\MVAGV}[1]{\MVA_{*,{#1}} (\Gamma)}

\newcommand{\MVAGC}{\MVAa{*} (\Gamma,\psi)}
\newcommand{\MVAGCa}[1]{\MVA_{#1} (\Gamma,\psi)}
\newcommand{\MVAGCb}[2]{\MVA_{#1} (\Gamma,\psi ; {#2})}
\newcommand{\MVAGCc}[2]{\MVA_{#1} (\Gamma,{#2})}
\newcommand{\MVAGCV}[1]{\MVA_{*,{#1}} (\Gamma,\psi)}

\newcommand{\MVBG}{\MVBa{*} (\Gamma ')}
\newcommand{\MVBGa}[1]{\MVB_{#1} (\Gamma ')}
\newcommand{\MVBGas}[1]{\MVB_{#1}^{{\mathrm{sym}}} (\Gamma ')}
\newcommand{\MVBGaa}[1]{\MVB_{#1}^{{\mathrm{skew}}} (\Gamma ')}
\newcommand{\MVBGb}[2]{\MVB_{#1} (\Gamma ' ; {#2})}
\newcommand{\MVBGbs}[2]{\MVB_{#1}^{{\mathrm{sym}}} (\Gamma ' ; {#2})}
\newcommand{\MVBGba}[2]{\MVB_{#1}^{{\mathrm{skew}}} (\Gamma ' ; {#2})}
\newcommand{\MVBGCa}[1]{\MVB_{#1} (\Gamma ' , \psi ' )}
\newcommand{\MVBGCas}[1]{\MVB_{#1}^{{\mathrm{sym}}} (\Gamma ' , \psi ' )}
\newcommand{\MVBGCaa}[1]{\MVB_{#1}^{{\mathrm{skew}}} (\Gamma ' , \psi ' )}
\newcommand{\MVBGCb}[2]{\MVB_{#1} (\Gamma ' , \psi ' ; {#2})}
\newcommand{\MVBGCbs}[2]{\MVB_{#1}^{{\mathrm{sym}}} (\Gamma ' , \psi ' ; {#2})}
\newcommand{\MVBGCba}[2]{\MVB_{#1}^{{\mathrm{skew}}} (\Gamma ' , \psi ' ; {#2})}

\newcommand{\MVCG}{\MVCa{*} (\Gamma ')}
\newcommand{\MVCGa}[1]{\MVC_{#1} (\Gamma ')}
\newcommand{\MVCGb}[2]{\MVC_{#1} (\Gamma ' ; {#2})}
\newcommand{\MVCGCa}[1]{\MVC_{#1} (\Gamma ' , \psi ' )}
\newcommand{\MVCGCb}[2]{\MVC_{#1} (\Gamma ' , \psi ' ; {#2})}


\title{The Skew-growth function on\\ the
 Monoid of square matrices}

\author{Kyoji Saito}

\begin{abstract}
\!\!\!\!\!We study an elementary divisibility theory for the monoid
 $\M(n,R)^\times\!\!$ $:=\!\! \{X\! \in\! \M(n,R)\! \mid \! \det(\!X)\!
 \not =\! 0\!\}$, where $R$ is a principal ideal domain and $\M(n,R)$ is
 the ring of $n$-by-$n$ matrices with coefficients in $R$. We prove that
 {\it any finite subset of $\M(n,R)^\times$ 
  has the right least common multiple} up to a left unit factor.

As an application, we consider the signed generating series, denoted by 
{\footnotesize $N_{\M(n,R)^\times,\deg}(t)$} and called the skew-growth function, of least common multiples
 of all finite sets of irreducible elements of $\M(n,R)^\times$, if $R$ is residue finite.
Then, using the divisibility theory, 
we show  the Euler product decomposition:

\medskip
\noindent
\ \ \ \ \ \ {\footnotesize $N_{\M(n,R)^\times\!,\deg}(\exp(\!-\!s))\! =\!\! \!\! \underset{p\in \!\{\! \text{primes of} R\!\}}{\prod}\!\!\!\! \!\! (1\!-\!N(p)^{-s})(1\!-\!N(p)^{-s\!+\!1}\!)\cdots(1\!-\!N(p)^{-s\!+\!n\!-\!1}\!) $}

\smallskip
\noindent
Here 
 $N(p):=\#(R/(p))$ is the {\it absolute norm} of $p\!\in\! R$
 {\tiny (there is an unfortunate coincidence of notation "$N$" for the  absolute norm and for the skew growth function [S4]).
 }
\end{abstract}

\maketitle 

\vspace{-1.0cm}
\tableofcontents
 \vspace{-1.0cm}
\footnote
{{\bf \small Acknowledgement:} The author is grateful to professors
Akio Fujii,  Masatoshi Suzuki, Satochi Kondo and Keiich Watanabe  for
informing the author some references and to Kazuhiko Kurano for a proof
of Fact in \S6. He also expresses  his gratitude to Scott Carnahan and to the referee
for the careful reading of the manuscript and valuable suggestions.
}

\vspace{-0.8cm}

 \section{\bf \large     Introduction} 


To any pair $(\M,\deg)$ of a cancellative monoid $\M$ and a degree map
 $\deg$ on $\M$, we associate the  {\it skew growth function}
 $N_{\M,\deg}(t)$ [S4]  in order to study certain thermo-dynamical limit
 functions [S2,3]. In the present paper, we study a particular example
 of a monoid  $\M(n,R)^\times\!:=\! \{ X\in \M(n,R) \mid
 \det(X)\not=0\}$ of square matrices of size $n\in \Z_{>0}$ with
 coefficients in a
 residue finite principal ideal domain $R$.   

For the purpose, we develop a divisibility theory of the monoid in a style  similarly to Artin monoids [B-S],\! and  show the existence of the right/left least common multiple for any finite subset of $\M(n,R)^\times$. 
Then the skew growth functions turns out to be a signed generating
series of the least common multiples of irreducible elements [S4]. Using further a phenomenon on the divisibility theory of $\M(n,R)^\times$, called  the  {\it jumps of  levels of irreducible elements} in  $\M(n,R)^\times$, 
 we  show that the skew-growth function  $N_{\M(n,R)^\times}(\exp(-s))$
 decomposes into an Euler product.

Let us explain this more precisely. 
For two elements $A,B\in
\M(n,R)^\times$, we say, as usual,  {\it $A$ divides $B$ from the left} or {\it $B$ is a right-multiple of $A$}, 
denoted by $A|_lB$,
if there exists $C\!\in\! \M(n,R)^\times$ such that $AC\!=\!B$. We say
$A$ is {\it left equivalent} to $A'\!\in \! \M(n,R)^\times$, denoted by
$A\! \sim_l\!  A'$, if $A|_lA'$ and $A'|_lA$ (which is equivalent to
an existence of an invertible matrix $E\in \GL(n,R)$ such that
$A'=AE$). Then the  division
relation  $A|_lB$ depends only on their left equivalence classes of $A$
and $B$, denoted by $[A],[B]$, in $\M(n,R)^\times/\!\! \sim_l$. 
In \S3, we introduce a normal form for each equivalence class, 
and denote by $M_n$ the set of all normal forms.

For a given finite
subset of $J$ of $\M(n,R)^\times$: i) we call an element
$L\in\M(n,r)^\times$ a {\it left-common multiple} of $J$ if $X|_l L$
for all $X\in J$ and ii) we call a left common multiple $L$ of $J$ a {\it least
left-common multiple of $J$} if it divides any (other) left-common multiple
of $J$ (for simplicity, we shall omit ``left''). 
The first main result of the present paper (\S5 Theorem 4) is the
following existence theorem.

\medskip
\noindent
{\bf Theorem 4.} {\it Any finite subset of $\M(n,R)$ admits a
least common multiple.}

\medskip
We shall denoted by $\LCM(J)$ the normal form of the unique equivalence
class of all least common multiples of $J$. 
The proof of
Theorem 4 is reduced to the case when $J$ consists only of two
elements, say $X$ and $Y$. In \S4 Theorem 3, we prove:

\medskip
\noindent
{\bf Theorem 3.} {\it  
There exists a unique map   
\vspace{-0.1cm}
 \[
 \sigma : \M(n,R)^\times \ \times\ M_n
 \to M_n , \ \  (Y,X) \mapsto \sigma_Y(X)
\vspace{-0.1cm}
\]
 such that  for any   $X\in M_n$ and $Y,Z\in  \M(n,R)^\times$, 
one has the equivalence 
\vspace{-0.1cm}
\[
 (*) \qquad\qquad \qquad\qquad \qquad
X\  |_l \ YZ    \quad     \Longleftrightarrow  \quad \sigma_Y(X)\  |_l\ Z  .
\qquad\qquad \qquad\qquad\qquad
\]
}
\ \ Then, $\LCM(X,Y)$ is given by the normal form of $Y\sigma_Y(X)$.\footnote{
\ We remark that Theorem 3 is formulated almost parallel with the key Lemma 3.1 in the
divisibility theory in Artin monoids ([B-S]) with respect to a
dictionary: $X\!\in\! M_n \!\leftrightarrow \! a\in I\!=$\{generators\},
$\sigma_Y(X) \!\leftrightarrow \! b$, $Y \! \leftrightarrow \! C$, $Z\!
\leftrightarrow \! D$, except for the difference that  Theorem 3 claims an equivalence "$\Leftrightarrow$" whereas  Lemma 3.1 claims only one implication "$a|_lCD\Rightarrow b|_lD$".
} 
The proof of Theorem 3 in \S3
 involves with irreducible decompositions of $X$ and $Y$ (Appendix).

\medskip
We switch our attention to the growth and skew-growth functions of the
monoid $\M(n,R)$, when $R$ is residue finite, i.e.\
$\#(R/(m))\!<\!\infty$ for all $m\!\in\! R\!\setminus\!\{0\}$. Namely,
using the absolute norm given by $N\! :\! R\! \setminus\! \{0\} \! \to\!  \Z_{>0},  \ m\! \mapsto\! N(m)\! :=\! \#(R/(m))$, we define the  degree  map:  
$ 
 X\! \in\!  \M(n,R)^\times \! \mapsto\!  \deg(X)\! :=\! \log(N(\det(X))) \! \in\! \R_{\ge0}.
$
Then, the growth and the  skew-growth functions (in simplified form [S4]) are given by
\vspace{-0.1cm}
\[
\begin{array}{rlll}
 P_{\M(n,R)^\times,\deg}(t) &:=& \sum_{[X]\in \M(n,R)^\times/\!\sim_l} t^{\deg([X])}\\
\vspace{-0.05cm}
 N_{\M(n,R)^\times,\deg}(t) &:=& \sum_{J: \text{ 
finite subset of } I_0} (-1)^{\#J} \ t^{\deg(\LCM(J))}
\end{array}
\]
where $I_0\!:=\!$ 
{\{left equivalence classes of all irreducible elements in $\M(n,R)^\times$\}}.  
Using the level structure on $\M(n,R)^\times$ in \S3, it is easy to get  the expression
$P_{\M(n,R)^\times,\deg}(\exp(\!-\!s))\! =\! \zeta_R(s)\zeta_R(s\!-\!1) \cdots \zeta_R(s\!-\!n\!+\!1)$,
where $\zeta_R(s)$ is the Dedekind zeta-function for $R$ (c.f.\ [Si][Ko]). 
Combining this with the inversion formula 
$P_{\M(n,R)^\times,\deg}(t) \cdot N_{\M(n,R)^\times,\deg}(t)\! =\!1$ ([S4]) and the Euler product 
of $\zeta_R(s)$, 
we get 
\[
\begin{array}{rlll}
\vspace{-0.05cm}
 N_{\M(n,R)^\times,\deg}(\exp(-s)) & =& \!\!\!\!\!\!\!\!\!\! \underset{p: \text{ prime of }R/units}\prod\!\!\!\!\!\!\!\!\!\!\!\!
(1\!-\!N(p)^{-s})
(1\!-\!N(p)^{-s\!+\!1})\!\cdots\! (1\!-\!N(p)^{-s\!+\!n\!-\!1}).
\end{array}
 \vspace{-0.1cm}
\]
However, as the second main result of the present paper, we prove this
formula directly in \S6,  using neither the inversion formula nor the Euler product of $\zeta_R(s)$, but using only the structure of the least common multiples on $\M(n,R)^\times$,
where, in the proof,  the jump of levels among $p$-irreducibles, introduced in \S5, is used
essentially to show some big cancellation of terms (see \S6, {\bf 7)}).
\footnote
{\ It is curious to compare this case with the skew growth
functions of Artin monoids, which are, conjecturally, irreducible polynomials over $\Z$ up to a factor $1\! -\!t$ ([S1]).}


\section{\bf \large     Monoid $\M(n,R)^\times$ and its irreducible elements}  
Let $R$ be a principal ideal domain. For any given positive integer $n\!\in\! \Z_{>0}$, consider the set of all square
matrices of size $n$ with non-zero determinant:
\[
 \M(n,R)^\times:=\{ X\in \M(n,R) \mid \det(X)\not=0 \}.
\]
The set $\M(n,R)^\times$ forms a monoid (i.e.\ a semi-group with
the unit $1_n$) with respect to the matrix product. 
Since $\M(n,R)^\times$ is embedded into
the group $\GL(n,\mathcal{F}(R))$ for $\mathcal{F}(R)=$the fractional field
of $R$, the monoid is cancellative,  that is, $AXB\!=\!AYB$ implies $X\!=\!Y$ for all $A,B,X,Y\in \M(n,R)^\times$. 

The set of all invertible elements in  $\M(n,R)^\times$ is given by
\[
\GL(n,R):=\{ X\in \M(n,R) \mid \det(X)\in \mathcal{E} \},
\]
where $\mathcal{E}$ is the unit group of $R$. An element $X\in \M(n,R)^\times$ is called {\it irreducible} if $X=YZ$
for $Y,Z\in  \M(n,R)^\times$ implies either $Y$ or $Z$ belongs to
$\GL(n,R)$. 

Let us show an elementary fact, which we use constantly in
the present paper.

\begin{lemm} An element $X\!\in\! \M(n,R)^\times$ is irreducible \ $\Leftrightarrow$  \
$\det(X)\! \in\! R$
is a  prime.
\end{lemm}
\begin{proof} 
Suppose $\det(X)$ is prime in $R$. If $X\!=\!YZ$ then 
 $\det(X)\!=\!\det(Y)\det(Z)$ and hence, either $\det(Y)$ or
 $\det(Z)$ belongs to $\mathcal{E}$, and either $Y$ or $Z$ belongs to
 $\GL(n,R)$.  Let us show the converse. Since,  for a principal ideal domain $R$,  any double coset in $\GL(n,R)\backslash \M(n,R)/\GL(n,R)$ can
 be presented by a diagonal matrix. So, consider a diagonal matrix $X$. 
If, either more than two diagonal entries of $X$ are non unit, or a diagonal
 entry of $X$ has more than two prime factors, then $X$ is
 reducible. That is, if $X$ is irreducible, then $\det(X)$ is a prime.
\end{proof}

\noindent
{\bf Definition.} Let $p$ be a prime element of $R$.
An element $X\in \M(n,R)^\times$  is called $p$-irreducible if $\det(X)$ is equal to $p$ up to a unit factor.
\medskip

\noindent
{\bf Remark.} 
Irreducible decompositions (non unique) of elements of $\M(n,R)^\times$
are studied in \S7 Appendix. We use them in the proof 7.\ of main Theorem 3 in \S4.

\bigskip
We denote $X|_l Y$ for $X,Y\!\in\! \M(n,R)^\times$, if there exists $Z\in \M(n,R)^\times$ such that $XZ\!=\!Y$, and we say that $X$ divides $Y$ from the left or $Y$ is a right multiple of $X$.

Define the {\it left-equivalence} $X\!\sim_l\!Y \!\Leftrightarrow_{\mathrm{def}} X|_lY \ \&\ Y|_lX$ ($\Leftrightarrow X\!=\!YE$ for a {\small $E\! \in\!\GL(n,R)$} due to cancellativity of {\small $\M(n,R)^\times$}), 
and denote by $[X]_l$ or by $[X]$ 
the left- equivalence class of an element $X\! \in\! \M(n,R)^\times$. That is, $[X]_l\! =\! X\!\cdot\! \GL(n,R)$, and 
\[ 
\begin{array}{rcl}
 \M(n,R)^\times/\!\!\sim_l & =& \M(n,R)^\times/\GL(n,R) ,\\
\end{array}
\]
where RHS is the quotient set by the right action of $\GL(n,R)$. 
Since the left-equivalence preserves the  left-division relation 
 (i.e.\  $X\sim_lX'$, $Y\sim_lY'$ and $X|_lY$ implies $X'|_lY'$), the quotient set $\M(n,R)^\times/\!\!\sim_l$ 
 naturally carries {\it poset structure} induced from the left-division relation: $[X]_l\le_l[Y]_l
\Leftrightarrow_{\rm def} X|_lY$.   Using the poset
structures, irreducible elements are characterized as follows. 

\medskip
\noindent
{\bf Fact.} {\it An element $X\in \M(n,R)^\times$ is irreducible if and only if $[X]_l$ is a minimal element in $(\M(n,R)^\times /\!\!\sim_l) \setminus \{[1_n]_l\}$ with respect to the poset structure $\le_l$.}

\medskip
\noindent
{\bf Remark.} 
Similar to the above, we can introduce the right division relation, the right equivalence relation on $\M(n,R)^\times$ and the poset structure on $\M(n,R)\!^\times/\!\!\sim_{r}\!=\!\GL(n,R)\!\setminus\! \M(n,R)\!^\times$.  
But, in the present paper, we study only $\M(n,R)\!^\times/\!\!\sim_{l}$, since one has a poset isomorphism:  $\M(n,R)\!^\times/\!\sim_{r} \ \equiv\  \M(n,R)\!^\times/\!\!\sim_{l}, [X]\mapsto [^tX]$.


\section{\bf \large     Normal form for the classes of  $\M(n,R)^{\times}/\!\sim_{l}$}
 We keep notation in \S2 so that $R$ is a principal ideal domain and $\mathcal{E}$ is the unit group of $R$.
We define normal forms for elements of the posets $\M(n,R)^{\times}/\!\sim_{l}$ for $n\in\Z_{\ge1}$. 
To this end, we fix, once and for all, a subset
$|R|\subset R\setminus\{0\}$ and subsets $R(m)\subset R$  for each $m\in |R|$, for which the
following natural bijections hold:
\[
\qquad  |R|  \simeq  (R\setminus\{0\})/\mathcal{E}  \qquad \text{and}  
\qquad  R(m)  \simeq  R/(m) \quad \text{for } m\in |R|,
\]
where  $(m)$ is the ideal in $R$ generated by $m$. 
Without loss of generality, we assume that 1) $M$ is
multiplicative by choosing the representatives for prime elements first, and
2)  the class of 0 (reps.\ 1) in $R/(m)$ is presented by $0$ (reps.\ 1)
in $R(m)$. 

\smallskip
\noindent
{\bf Eg.} Let $R=\Z$. Then, we choose $|R|=\Z_{>0}$ and $R(m)=[0,|m|-1]\cap\Z$.

\medskip
Depending on the choices of $|R|$ and $R(m)$, we introduce a subset of $\M(n,R)^\times$:
{\small
\[
 M_n:=\left\{
\begin{bmatrix}
m_1    & 0      & 0    & \cdots& 0\\
d_{21} & m_2    & 0    & \cdots& 0\\
d_{31} & d_{32} & m_3  & \cdots& 0\\
  *    &    *   &  *   & \cdots& 0\\
d_{n1} & d_{n2} & d_{n3} & \cdots& m_n 
\end{bmatrix}
\biggm. 
\begin{array}{ll}
m_1,m_2,\cdots,m_n\in |R|, \\ 
\\
d_{i1},d_{i2},\cdots,d_{i (i\!-\!1)} 
\in R(m_i) \\
\qquad\quad \text{ for }i=2,\cdots,n
\end{array}
\right\}
\]
\vspace{-0.2cm}
}

\begin{lemm} Every right $\GL(n,R)$-orbit (= a left equivalent class) in $\M(n,R)^\times$ intersects with
 the set $M_n$ at a single element. That is,  the projection  $\M(n,R)^\times\to \M(n,R)^\times/\GL(n,R)$  
induces a bijection
\[
 M_n \ \simeq \  \M(n,R)^\times/\!\sim_l \ .
\]
\end{lemm}
\begin{proof}  This is shown by an induction on $n\in\Z_{>0}$.

 Case $n\!=\!1$ is shown by  $\M(1,R)^\times/\!\!\!\sim\!=\!(R\!\setminus\!\{0\})/\mathcal{E}\!\simeq\! |R|\!=\!M_1$. 
Let $n\!>\!1$ and assume Lemma for $n\!-\!1$.
We first show that the  projection from $M_n$ is surjective.
Let $X\!\in\! \M(n,R)^\times$ and let ${\bf x}\!=\!(x_1,\cdots,x_n)\in R^n$ be its
 first row vector, which is non-zero by the determinant condition  $\det(X)\!\not=\!0$. Then, there exists  $m_1\!\in\! M$, which generates the ideal $(x_1,\cdots,x_n)$, and $A\in
\GL(n,R)$ such that ${\bf x}A\!=\!(m_1,0,\cdots,0)$.
 Hence, 
 we may choose a representative of the class $[X]_l$ to be of the
 form:
$
 \begin{bmatrix}
m_1\!\! &\! 0\\ 
*\! \!&\! X'
\end{bmatrix}
$
for $X\!'\in\! \M(n\!-\!1,\Z)^\times$. By our induction hypothesis, there exists $A'\!\in\! \GL(n\!-\!1,R)$ such that 
$
 \begin{bmatrix}
m_1 \!&\! 0\\ 
*\!&\! X'
\end{bmatrix}
\begin{bmatrix}
1&\! 0\\ 
0&\! A'
\end{bmatrix}
\!=\!
\begin{bmatrix}
m_1 \!&\! 0\\ 
 *\!&\! X''
\end{bmatrix}
$
where $X''$ is an element of $M_{n-1}$ whose diagonal is
 $(m_j)_{j=2}^{n}\in |R|^{n-1}$. Then 
we  find a column vector $[*']\in R^{n-1}$ such that
 $[*]\!+\!X''[*']=:[d']$ is a vector in {\small $\prod_{i=2}^n R(m_i)$}.
Applying a matrix of the form
\noindent 
$\begin{bmatrix}
1 \!&\!\! 0\\ 
*' \!&\!\! 1_{n\!-\!1}
\end{bmatrix}
$
from the right, we get the normal form
$\begin{bmatrix}
m_1 \!&\!\! 0\\ 
* \!&\!\! X''
\end{bmatrix}
\begin{bmatrix}
1 \!&\!\! 0\\ 
*' \!&\!\! 1_{n\!-\!1}
\end{bmatrix}
\!=\!
\begin{bmatrix}
m_1 \!&\!\! 0\\ 
 d'  \!&\!\! X''
\end{bmatrix}$.

Next we show the injectivity of the correspondence. Let $X,Y\in M_n$ such that
 $X\sim_l Y$. Then $U:=X^{-1}Y$ is a lower triangular matrix in
 $\GL(n,R)$, whose diagonal entries are of $\E$.  Since $m^{-1}m'\!\in\!\E$ for $m,m'\!\in\! M$ implies $m^{-1}m'\!=\!1$, diagonals of $U$ are 1. This proves, in particular, the case for $n\!=\!1$. 

For $n\!>\!1$, restricting the equality $XU\!=\!Y$ to the two $(n\!-\!1)\times(n\!-\!1)$ principal sub-matrices forgetting either the first column and row or the last column and row, respectively, we see that parts of $X$ and $Y$  are left-equivalent. By the
 induction hypothesis, the corresponding $(n-1)\times(n-1)$ principal sub-matrices of
 $U$ are equal to the identity matrix. Thus, $U$ is equal to the identity matrix  $1_n$ of size $n$ up to the
 $(n,1)$-entry $u_{n1}$. Then the equality $XU=Y$ implies
 $x_{n1}+u_{n1}m_n=y_{n1}$. Since we have the normalization
 $x_{n1},y_{n1}\in R(m_n)$, we get $x_{n1}=y_{n1}$ and $u_{n1}=0$.
\end{proof}

\noindent
{\bf Definition 1.}  For a left equivalence class $[X]\in
\M(n,R)^\times/\!\!\sim_l$, we call the element in
$X\GL(n,R)\cap M_n$ the {\it normal form} of $[X]$.
{\it We shall often identify the class $[X]$ with its normal form, when there is no possibility of confusions.}

\noindent
{\bf 2.}  By the {\it diagonal part of} a class $[X]$, denoted by $\diag([X])$,
 we mean the diagonal part of its normal form which is an ordered sequence  $(m_1,\!\cdots\!,m_n)\!\in\! |R|^n$, where
  $m_i$  ($1\le i\le n$) is called the diagonal entry of $[X]$ of $i$th {\it level}.  

\noindent
 {\bf Notation.}
 {\small For\! a row vector}\!  ${\bf x}\!\in\! R^n$\!, {\small we define an element of}
 {\it pure of level} $1\! \le\! i \! \le \! n$ 
by
 \[
 M(i:{\bf x}):=\substack{  \text{\normalsize  the matrix obtained by substituting the}
  \\ \text{\normalsize   $i$th row of the identity matrix $1_n$ by  ${\bf x}$. }}
 \]

\noindent
{\bf Definition 3.}
If  $M(i:{\bf x})\!\in\! M_n$, i.e.\ ${\bf x}\!=\!(d_1,\!\cdots\!,d_{i-1},m,0,\!\cdots\!,0)$ for some $m\!\in\! |R|$ and
 $d_j\!\in\! R(m)$ $(1\! \le \! j\!  <\! i)$,   we call 
it  a {\it normal form  of level $i$ with diagonal $m$}. 

\noindent
{\bf 4.}  If the diagonal $m$ of a normal form $M(i:{\bf x})$ is a prime element, say $p$, in $R$,
 we call  $M(i:{\bf x})$ a {\it $p$-irreducible (normal form)} of level $i$.

\medskip
It is clear that any irreducible element of $\M(n,R)^\times$ is left equivalent to a unique irreducible normal form for a certain level $i$ ($1\le i\le n$).  We call $i$ the {\it level of the irreducible element}. For any prime element $p\in R$, the set of left equivalence classes of all $p$-irreducible elements is naturally bijective to the  set 
\[
\begin{array}{l}
I_{0,p}:= \bigsqcup_{i=1}^n \ \{ M(i:({\bf d},p,0,\cdots,0)) \mid {\bf d}\!=\!(d_1,\cdots,d_{i-1})\in (R(p))^{i-1}\}
\end{array}
\]
of all $p$-irreducible normal forms.  We shall sometimes confuse them.


\section{\bf \large     Left divisibility theory}

We develop  an {\it elementary divisibility theory for} $\M(n,R)^\times$ in a style similar to the divisibility theory for Artin monoids ([BS,\S3], see Footnote 2).
The main result is formulated in Theorem 3.
In order to state the result, let us recall notations: $R$ is a principal ideal domain, $|R|$ is  a subset of $R$ s.t.\  $|R| \! \simeq \! (R  \! \setminus \! \{0\} )/\{\text{units}\}$, and $M_n$ is the set of normal forms in $\M(n,R)^\times$ s.t.\  $M_n \simeq \M(n,R)^\times\! /  \GL(n,R)$ (\S3).  

\begin{theo}  
There exists a unique map   
 \[
 \sigma : \M(n,R)^\times \ \times\ M_n
 \to M_n , \ \  (Y,X) \mapsto \sigma_Y(X)
\]
 such that  for any   $X\in M_n$ and $Y,Z\in  \M(n,R)^\times$, 
one has the equivalence 
\[
 (*) \qquad\qquad \qquad\qquad \qquad
X\  |_l \ YZ    \quad     \Longleftrightarrow  \quad \sigma_Y(X)\  |_l\ Z  .
\qquad\qquad \qquad\qquad\qquad
\]
The map $\sigma$ satisfies further the following {\bf 1)}-{\bf 4)}.

\medskip
\noindent
{\bf 1) $\M(n,R)^\times)$-action on $M_n$.}  The map $\sigma$ defines an opposite left action $\sigma_Y$ of $Y\in \M(n,R)^\times$ on the set $M_n$ with the  fixed point $1_n$. That is,
\[
\sigma_{1_n}=id_{M_n}, \quad     \sigma_{Y_2Y_1}=\sigma_{Y_1}\circ \sigma_{Y_2} \quad \text{and}\quad  \sigma_Y(1_n)=1_n
\]
for any $Y,Y_1,Y_2\!\in\! \M(n,R)^\times$. 

\medskip
\noindent
{\bf 2) $\GL(n,R)$-action on $M_n$.} The map $\sigma$ is compatible with the actions of $\GL(n,R)$ on $\M(n,R)^\times$ from right and on $M_n$ from left. That is,  for any $Y\in \M(n,R)^\times$ and $E\in \GL(n,R)$, we have 
\[
\sigma_{YE} = [E^{-1}]\sigma_Y 
\]
where $[E^{-1}]$ denote the left action on the set $M_n$ induced from
 the left action of $E^{-1}$ on the set of left equivalence classes
 through the identification {\bf Lemma 2}.

\medskip
\noindent
{\bf 3) Reciprocity.}  There exists a map 
$ u:  M_n \times M_n \rightarrow 
\mathrm{U}(n,R)\!:=\!\{$lower triangular matrices in $\M(n,R)$ whose diagonals are $1\}$ such that,  for $X,Y\!\in\! M_n$, we have
\[
X\sigma_X(Y) =Y\sigma_Y(X) \cdot u(X,Y).
\]
(here $u(X,Y)$ is of pure of level equal to the maximum of levels of $X$ and $Y$.
  
\medskip
\noindent
{\bf 4) Monotonicity.}  For any $X\in M_n$ and $Y\in \M(n,R)^\times$,  one has
\[
   \frac{\det(X)}{\det(\sigma_Y(X))} \ \in \ |R|  \quad (\ \simeq (R\setminus \{0\})/\mathcal{E}\ )
   \]
In particular, if $X$ is a $p$-irreducible element, then $\sigma_Y(X)$
 is equal either to 1 or to a $p$-irreducible element.  
 Precisely, $\sigma_Y(X)$ is a $p$-irreducible of the
 same level as 
 $X$ if and only if an irreducible decomposition of $Y$ does not
 contain a $p$-irreducible element of the same level as $X$.
More over, let $X\!=\!M(i\!:\!({\bf d},p,{\bf 0}))$ (resp.\
 $Y\!=\!M(j\!:\!({\bf e},q,{\bf 0}))$) be normal forms of $p$- (reps.\
 $q$-) irreducible elements of level $i$ (reps.\ $j$) with
 $X\!\not=\!Y$. Then, $\sigma_Y(X)$ is $p$-irreducible such that   
\[ 
\text{level of } \sigma_Y(X)  =
\begin{cases}
i   &\text{ if }  i\not=j \ \text{ or } \ p\not=q\\
\max\{k\mid d_k\not\equiv e_k \bmod{p},\ 1\le k<i\} &\text{ if }  i=j \text{ and } p=q.
\end{cases}
\]
 
\end{theo}
\begin{proof}
Let us, first, give an overview of the proof. 

The proof is divided into nine steps {\bf 1} -  {\bf 9.}
In {\bf 1.}, we show the uniqueness of the map $\sigma$ satisfying $(*)$
 (if it exists). Then in {\bf 2.-4.}, using the uniqueness property, we
 show that the properties 1), 2) and 3) of $\sigma$ are deduced from
 $(*)$. In {\bf  5.}, we show a criterion for a lower triangular matrix
 to be divisible from left by a $p$-irreducible element. Then, using the
 criterion {\bf 5.}, we show in {\bf 6.}, the  existence of
 $\sigma_Y(X)$ when $X$ and $Y$ are irreducible normal forms.  Then, in
 {\bf 7.},  applying the composition rule in {\bf 1)} to {\bf 6.}\ repeatedly for an irreducible
 decomposition of $Y$, we show the existence of
 $\sigma_Y(X)$ for general $Y\in \M(n,R)^\times$. Then, using 3) Reciprocity, we exchange the role of $X$ and $Y$. To the
 general $Y$, we again apply the rule in {\bf 1)} for an irreducible decomposition of $X$,
 and we obtain $\sigma_X(Y)$ for all $(X,Y)\in M_n\times M_n$.
 Finally, the property
 {\bf 4)} is shown in {\bf 8.}\ and {\bf 9.}\ by induction on the number
 of irreducible factors of $X$ and $Y$, where the essential case is when
 $X$ and $Y$ are irreducible, discussed in {\bf 6.}

The core of the proof is {\bf 5., 6.}\ where we use matrix expression of the monoid  $\M(n,R)^\times$, whereas the other parts {\bf 1., 2., 3., 4.}\ and {\bf 7.}\ of the proof are general properties for any cancellative monoids, and {\bf 8.} and {\bf 9.} are applications of {\bf 6.}

\medskip
{\bf 1.}  {\it For a given pair $(X,Y)\in M_n\times \M(n,R)$, if there exists  $\sigma_Y(X)\in M_n$  satisfying the condition $(*)$, then it is unique. }
\begin{proof} 
Suppose  there are two elements $\sigma,\sigma' \!\in\! \M(n,R)$ satisfying $(*)$. It implies, in particular,  
$\sigma|_l Z \! \Leftrightarrow \! \sigma' |_lZ$  for any $Z \! \in\!  \M(n,R)$. Then, by choosing $Z$ to be $\sigma$ and $\sigma'$, we get $\sigma' |_l \sigma$ and $\sigma |_l \sigma'$. That is, $\sigma$ and $\sigma'$ are left equivalent, i.e.\  $[\sigma]\! =\! [\sigma']$.  
\end{proof}

{\bf 2.}\!  
{\it Let a map $\sigma_Y\!:\!M_n\!\to\! M_n$ satisfy
 the condition $(*)$ for  $Y\!\in\! \M(n,R)^\times$\!. Then
 for any $E\! \in\! \GL(n,R)$, the map $[E^{-1}]\sigma_{Y}: X\in M_n\mapsto
 [E^{-1}]\sigma_Y(X))\! \in\! M_n$ satisfies the condition $(*)$ for $YE$. That is, $\sigma_{YE}$ exists and is equal to $[E^{-1}]\sigma_Y$.}
\begin{proof}
\!We have:  $X|_l  YEZ\! \Leftrightarrow\! \sigma_Y(X) |_l EZ\!  \Leftrightarrow\! E^{-\!1}\sigma_Y(X) |_l Z$.\  This means that  $[E^{-\!1}]\sigma_Y$ satisfies the condition $(*)$ for $\sigma_{YE}$. Then, the uniqueness {\bf 1.} implies the result.  
\end{proof}

{\bf 3.}  {\it Suppose that there exist the maps $\sigma_{Y_1}$ and $\sigma_{Y_2}$ for  $Y_1$ and $Y_2\in \M(n,R)$  satisfying the condition $(*)$, respectively. Then, the composition $\sigma_{Y_1}\circ \sigma_{Y_2}$ satisfies the condition $(*)$ for $Y_2Y_1$.  That is, $\sigma_{Y_2Y_1}$ exists, and is equal to $\sigma_{Y_1}\circ \sigma_{Y_2}$.}
\begin{proof}   We have:\  $X\ |_l\  Y_2Y_1Z \ \ \Leftrightarrow\ \  \sigma_{Y_2}(X)\ |_l\ Y_1Z\ \  \Leftrightarrow\ \  \sigma_{Y_1}(\sigma_{Y_2}(X))\ |_l\ Z$. 
\end{proof}

\medskip
{\bf 4.} {\it If there exists $\sigma_Y(X)$ satisfying $(*)$ for $X,Y\in M_n$, then there exists $\sigma_X(Y)$ satisfying $(*)$ for $Y,X$, and an element $u(X,Y)\in \mathrm{U}(n,R)$ satisfying the equation}  \[
X\sigma_X(Y) =Y\sigma_Y(X) u(X,Y).
\]
\begin{proof}  That $\sigma_Y(X)$ satisfies the condition $(*)$ for the pair $(X,Y)$ implies, in particular (by choosing $Z=\sigma_Y(X)$), $X|_l (Y\sigma_Y(X))$. So, put $Y\sigma_Y(X)=XW$ for a $W\!\!=$ a lower triangular matrix in $\M(n,R)^\times$.
Let us show that the class $\sigma:=[W]\in M_n$ (recall Lemma 2) satisfies the condition $(*)$ for the pair $(Y,X)$. That is, we need to show the equivalence $Y|_l XZ \Leftrightarrow \sigma|_lZ$ for any $Z\in \M(n,R)^\times$. The implication "$\Rightarrow$" follows, since $Y|_lXZ$ implies an existence of $V\in \M(n,R)^\times$ with $XZ=YV$. Then, $X|_lYV$ and  $(*)$ for $(X,Y)$ implies $\sigma_Y(X) |_l V$. So, put $V=\sigma_Y(X) U$ for $U\in \M(n,R)^\times$. Then, $XZ=Y\sigma_Y(X)U=XWU$. Cancelling $X$ from left, we obtain $Z=WU$. That is, $Z$ is left divisible by the class of $W$, i.e.\ by $\sigma$.  The opposite implication "$\Leftarrow$" follows, since $\sigma|_lZ$ implies an existence of $T\in \M(n,R)^\times$ such that $Z=\sigma T$, and, hence $XZ=X\sigma T=XWET=Y\sigma_Y(X)ET$ for some $E\in \GL(n,R)$. That is, $XZ$ is left-divisible by $Y$.
\end{proof}

{\bf 5.} {\it A lower triangular matrix 
$Z\!=\!\begin{bmatrix} Z' \!\!&\! 0\! &\! 0 \vspace{-0.15cm}\\ 
{\bf z}  \!\!\!&\! v \!&\!0 \vspace{-0.05cm}\\ 
\vspace{-0.1cm}
 *\! &\! * \! &\!  *\! 
 \vspace{-0.1cm}
 \end{bmatrix}$ with $Z'\in \M(i\!-\!1,R)$,  ${\bf z}\!\in\!
 R^{i\!-\!1}$ and $v\!\in\! R$, 
is divisible by a $p$-irreducible element $X\! =\! M(i\!:\!({\bf d},p,{\bf 0}))$ for ${\bf d}\! \in\! R^{i\!-\!1}$\! of level $1\!\le\! i\!\le \! n$ from left, 
if and only if they satisfy }
\[
p\mid v \qquad  \text{and} \qquad   {\bf z}\equiv {\bf d}Z'\ \bmod {p}.
\]
\begin{proof} 
 Since $X^{-1}\!=\!M(i:(-\frac{\bf d}{p}, \frac{1}{p},{\bf 0}))$, we
 observe that $X^{-1}Z$ is equal to $Z$ except for the $i$th row, where the $i$th row is given by $(\frac{1}{p}({\bf z}\!-\!{\bf d}Z'), \frac{v}{p},{\bf 0})$. 
\end{proof}

{\bf 6.} We study the case when $X$ and $Y$ are irreducible in $\M(n,R)^\times$. This part is the essential part of the whole proof of Theorem 3.

\medskip
\noindent
{\bf Assertion.}
 {\it Let $X$ and $Y$ be a $p$-irreducible element of level $i$ and  a
 $q$-irreducible element of level $j$ for primes $p,q\! \in\!\! M$ and
 $1\!\le\! i,j \! \le \! n$. Then, either $X\!=\!Y$ and
 $\sigma_Y(X)\!=\!1_n$, or there exists a $p$-irreducible element
 $\sigma_Y(X)$ satisfying  condition $(*)$ whose level is unchanged from
 that of $X$ except for the case $p\!=\!q$ and $i\!=\j$.} 
\begin{proof} The proof is divided into 4 cases.  

{\bf Case i)}  \quad  $i<j$.

\smallskip
Since $Y$ is of level $j$,  $YZ$ for any $Z\in \M(n,R)$ is a matrix which coincides with $Z$ from 1 to $j\!-\!1$ rows. On the other hand, the divisibility of $YZ$ (resp. Z) by $X$ from the left is determined by the row vectors of $YZ$ (resp. $Z$) from 1 to $i$th. That is, we have the equivalence $X|_l YZ \Leftrightarrow X|_l Z$. That is, we have 

\smallskip
\centerline
{ $\sigma_Y(X)=X$. }
This completes the proof for the case when $i<j$. \qquad $\Box$

\medskip
{\bf Case ii)}   \quad   $i=j$ and $p=q$.

\smallskip

This is the most intricate and subtle case.

\smallskip
If $X=Y$, we have $\sigma_{Y}(X)= 1_n$. 
Suppose $X\not=Y$, and let $X=M(i:({\bf d},p,{\bf 0}))$ and $Y=M(i:({\bf e},p,{\bf 0}))$ 
for ${\bf d}, {\bf e} \in (R(p))^{i-1}$ with ${\bf d}-{\bf e}\not=0$. 
Let  the $i$-principal sub-matrix of $Z$ is of the form  $=\!\begin{bmatrix} Z' \!\!&\! 0\\ {\bf z}  \!\!\!&\! v \end{bmatrix}$ with $Z'\in \M(i\!-\!1,R)$,  ${\bf z}\!\in\! R^{i\!-\!1}$ and $v\!\in\! R$.  Then, the $i$-principal sub-matrix of $YZ$ is of the form  $\!\begin{bmatrix} Z' \!\!&\! 0\\ {\bf e}Z'\!+\!p{\bf z}  \!\!\!&\! pv \end{bmatrix}$.
 Then the criterion in {\bf 5.} says that 
\[
\begin{array}{rcl}
X|_l YZ  \quad&  \Leftrightarrow & \quad   p|pv  \text{\quad and \quad}  {\bf e}Z'\!+\!p{\bf z} \equiv {\bf d}Z' \bmod{p}\\ 
& \Leftrightarrow &\quad ({\bf e}-{\bf d})Z'\equiv 0 \bmod{p}
\end{array}
\]

Let us find one particular  solution of the equations, satisfying  $v=1$ and ${\bf z}=0$.  
By the assumption $X\not=Y$, there is some 
\[
k:=\max\{1\le l < i \mid e_l-d_l\not\equiv 0\bmod {p}\}. 
\]
Then, we consider a $p$-irreducible element $W:=M(k:({\bf f},p,{\bf 0}))$ with ${\bf 0}\in R^{n-k}$, where ${\bf f}\in R(p)^{k-1}$ is  defined as: for $1\le l<k$, we solve the following equation  
\[
e_l-d_l +f_l(e_k-d_k)\equiv 0 \bmod{p}
\]
on $f_l\in R(p)$.  This is solvable since $e_k-d_k$ is prime to $p$ in  $R$. 
The $i$-principal sub matrix of $W$ is of the form $M(k:({\bf f},p,{\bf 0}))$  with ${\bf 0}\in R^{i-1-k}$ and satisfies the equation $({\bf e}-{\bf d})M(k:({\bf f},p,{\bf 0}))\equiv 0 \bmod{p}$. This means that $W$ is a solution of $X|_l YW$. Then, any $Z$ with $W|_l Z$  satisfies $X|_l YW|_l YZ$.

On the other hand, let us consider any lower triangular matrix $Z$ satisfying $X|_l YZ$. 
We want to show $W|_l Z$, where, according to {\bf 5.}, $W|_l Z$ if and only if 
\[
p\mid v''  \qquad  \text{and} \qquad   {\bf z''}\equiv {\bf f}Z''\ \bmod {p},
\]
where  the $k$-principal sub-matrix of $Z$ is of the form  $\begin{bmatrix} Z''\!\!&\! 0\\ {\bf z''} \!\!\!&\! v'' \end{bmatrix}\in \M(k,R)^\times$. 
Since $e_m-d_m\equiv 0\bmod{p}$ for $m$ with $k<m\le i-1$, the condition $X|_lYZ$ on $Z$, i.e.\ $({\bf e}-{\bf d})Z'\equiv 0 \bmod{p}$ on $Z$ can be rewritten as  $({\bf e''}-{\bf d''}) \begin{bmatrix} Z''\!\!&\! 0\\ {\bf z''}  \!\!\!&\! v'' \end{bmatrix}\equiv 0 \bmod{p}$, where  $({\bf e''}-{\bf d''}) $ is the row vector  consisting of the first $k$ entries of $({\bf e}-{\bf d}) $.  
Since, by the definition of ${\bf f}$, we have $({\bf e''}-{\bf d''}) \equiv(e_k-d_k)(-{\bf f}, 1) \bmod{p}$.  Then the condition $X|_lYZ$ on $Z$ can be further rewritten as 
$(e_k-d_k)(-{\bf f},1) \begin{bmatrix} Z''\!\!&\! 0\\ {\bf z''}  \!\!\!&\! v'' \end{bmatrix}\equiv 0 \bmod{p}$.
Since by the choice of $k$, $e_k-d_k$ is prime to $p$ so that we can
 divide the equality by $d_k-e_k$.   Then, this condition exactly
 implies $p\mid v''$ and $ {\bf z''}\equiv {\bf f}Z''\ \bmod {p}$.  That
 is, the condition $X|_lYZ$ implies the condition $W|_lZ$ (in fact, they are equivalent). 
Then, $W$ satisfies the property $(*)$, and we put
\[
\sigma_Y(X):=W=M(k:({\bf  f},p,{\bf 0})).
\]

This completes the proof for the case when $p=q$ and $i=j$.  $\Box$

\medskip
\noindent
{\bf Remark.} We have shown the latter half of {\bf 4)} for the case
 $i=j$ and $p=q$. In particular, the level $k$ of $\sigma_Y(X)$ is {\it strictly smaller} than the level $i$ of $X$ and $Y$.  We shall call this phenomenon the {\it \bf jump of levels} of $p$-irreducible elements.

\medskip
{\bf Case iii)}   \quad   $i=j$ and $p\not=q$.

\smallskip
Let $X=M(i:({\bf d},p,{\bf 0}))$ and $Y=M(i:({\bf e},q,{\bf 0}))$ for ${\bf d}\in R(p)^{i-1}$ and ${\bf e} \in R(q)^{i-1}$. Let  the $i$-principal sub-matrix of $Z$ be of the form  $\begin{bmatrix} Z' \!\!&\! 0\\ {\bf z}  \!\!\!&\! v \end{bmatrix}$ with $Z'\in \M(i\!-\!1,R)$,  ${\bf z}\!\in\! R^{i\!-\!1}$ and $v\!\in\! R$.  Then, the $i$-principal sub-matrix of $YZ$ is of the form  $=\!\begin{bmatrix} Z' \!\!&\! 0\\ {\bf e}Z'\!+\!q{\bf z}  \!\!\!&\! qv \end{bmatrix}$ with $Z'\in \M(i\!-\!1,R)$. 
 The criterion in {\bf 5.} says that 
\[
\begin{array}{rcl}
X|_l YZ  \quad&  \Leftrightarrow & \quad   p|qv  \text{\quad and \quad}  {\bf e}Z'\!+\!q{\bf z} \equiv {\bf d}Z' \bmod{p}\\
& \Leftrightarrow & \quad  p| v \quad \text{and}\quad ({\bf e}-{\bf d})Z' +q{\bf z}\equiv 0 \bmod{p}
\end{array}
\]
Let us give one particular solution $W=M(i:({\bf f},p,{\bf 0}))$, satisfying  $X|_lYW$. Namely, we put  $v=p$ and $Z'=1_{i-1}$, then, since $p$ and $q$ are prime, the equation $q{\bf z}\equiv {\bf d}-{\bf e} \bmod {p}$ on ${\bf z}$ has a unique solution  ${\bf f}\in R(p)^{i-1}$.   Then, obviously, for any $Z\in \M(n,R)^\times$ with $W|_l Z$, we get $X|_l YW|_l YZ$.

On the other hand, let us consider any lower-triangular matrix $Z$ satisfying $X|_l YZ$.
We want to show $W|_l Z$, where, according to {\bf 5.}, $W|_lZ$ if and only if 
\[
p\mid v \qquad  \text{and} \qquad   {\bf z}\equiv {\bf f}Z'\ \bmod {p},
\]
where the first condition $p|v$ is already satisfied.  Furthermore, substituting the relation ${\bf e}-{\bf d}\equiv-q{\bf f}$ in the condition $X|_lYZ$, we obtain $-q{\bf f}Z'+q{\bf z} \equiv0 \bmod{p}$.  Since $q$ is prime to $p$, we can divide this equality by $q$, and we obtain the condition for $W|_l Z$. 
That is, the condition $X|_lYZ$ implies the condition $W|_lZ$ (in fact, they are equivalent). 
Thus, $W$ satisfies the property $(*)$, and we put
\[
\sigma_Y(X):=W=M(i:({\bf  f},p,{\bf 0})).
\]

This complete the proof for the case when $p\not=q$ and $i=j$.  $\Box$

\medskip
{\bf Case iv)}  \quad   $i>j$.

\smallskip

Let $X=M(i:({\bf d},p,{\bf 0}))$ and $Y=M(j:({\bf e},q,{\bf 0}))$ 
for ${\bf d}\in R(p)^{i-1}$ and ${\bf e} \in R(q)^{j-1}$, where $p$ may or may not be equal to $q$.  Let $Z\in \M(n,R)^\times$ be any lower triangular matrix, whose  
 $i$-principal sub-matrix  is of the form  $\begin{bmatrix} Z' \!\!&\! 0\\ {\bf z}  \!\!\!&\! v \end{bmatrix}
 \in \M(i,R)^\times$ with $Z'\in \M(i\!-\!1,R)$,  ${\bf z}\!\in\! R^{i\!-\!1}$ and $v\!\in\! R$.  Then, the $i$-principal sub-matrix of $YZ$ is of the form  
 $\left(1_i+\begin{bmatrix} {\bf 0}
 \vspace{-0.1cm}
 \\ ({\bf e},q\! -\!1,{\bf 0})
 \vspace{-0.1cm}
\\ {\bf 0}\end{bmatrix}\right)
 \begin{bmatrix} Z' \!\!&\! 0\\ {\bf z}  \!\!\!&\! v \end{bmatrix} 
= \begin{bmatrix} Z' \!\!&\! 0\\ {\bf z}  \!\!\!&\! v \end{bmatrix} 
+\begin{bmatrix} {\bf 0}
\vspace{-0.1cm}
\\ ({\bf e},q\! -\! 1,{\bf 0})Z'
\vspace{-0.1cm}
\\ {\bf 0}\end{bmatrix}$, where 

1)  $({\bf e},q-1,{\bf 0})$ is a row vector  located  in the $j$th row.  Since $i>j$, the size $j$ of the vector $({\bf e},q-1)$ is strictly smaller than the size $i$ of  the matrix. 

2) ${\bf 0}$'s are zero matrices or zero vectors whose size depends on the place where they are located.  In particular, due to the inequality $i>j$, the ${\bf 0}$'s in the bottom row are non-empty. This implies that the $i$th row vector of $YZ$ is equal to that of $Z$ and is $({\bf z},v,{\bf 0})$.

 Then the criterion in {\bf 5.} says that 
\[
\begin{array}{rcl}
X|_l YZ  \quad&  \Leftrightarrow & \quad   p|v  \text{\quad and \quad}  {\bf z} \equiv ({\bf d}+ d_j({\bf e},q-1,{\bf 0}))Z' \bmod{p}\\ 
\end{array}
\]
Reversing the criterion {\bf 5.}, the last condition is equivalent to that $Z$ is divisible by 
$W:=M(i:({\bf d}+d_j({\bf e},q-1, {\bf 0}),p,{\bf 0}))$.  Clearly, $W$
 is a $p$-irreducible element (even if it is not yet a normal form because of the term $d_j({\bf e},q-1,{\bf 0})$), 

Thus, we put
\[
\sigma_Y(X):=\text{ the normal form of } W. 
\]

This completes the  proof of the case $i>j$, and, hence, that of 6. 
\end{proof}

{\bf 7.}
In \S7. Appendix Lemma 10, we show that for any element $Y\in \M(n,R)^\times$  with $\diag([Y])=(m_1,\cdots,m_n)$ and irreducible decompositions $m_i=\prod_{k=1}^{k_i} p_{i,k}$ ($i=1,\cdots,n$), there exists a unique decomposition $Y=(\prod_{i=1}^{n}\prod_{k=1}^{k_i}P_{i,k}) E$
where $P_{i,k}$ is a $p_{i,k}$-irreducible normal form of level $i$ and
 $E\in \GL(n,R)$. Then, for any $p$-irreducible normal form $X$,
 applying the composition rule in 1) and 2) of Theorem 3, we see that $\sigma_Y(X)$ is given by
\[
\begin{array}{c}
\sigma_Y(X):=\Big([E^{-1}] \prod_{i=n}^{1}\prod_{k=k_i}^{1}\sigma_{P_{i,k}} \Big) (X)
\end{array}
\]
where RHS means 1) act $\sigma_{P_{i,k}}$ on $X$ successively in the
 lexicographic order, 2) left act of $[E^{-1}]$ (use 1, 2, 3 and 6).  The result  is either a
 p-irreducible element or $1_n$. 

\smallskip
So far, we constructed $\sigma_Y(X)$ for a $p$-irreducible
 $X$. We want to construct it for arbitrary  $X\! \in\! M_n$ and  $Y\! \in\! \M(n,R)^\times$.
Due to {\bf 2)} or {\bf 2.}, it is sufficient to show the existence  for
 the case when $Y\in M_n$.  Then, due to {\bf 4.}Reciprocity, this is equivalent to
 show the existence of $\sigma_X(Y)$. For general $X\in \M(n,R)^\times$, taking an irreducible  decomposition 
 $X=(\prod_{i=1}^{n}\prod_{k'=1}^{k_i'}P_{i,k'}' ) E'$ in \S7 Appendix,
 and applying again the composition rule in 1) and 2), we get  
 \[
 \begin{array}{c}
\sigma_X(Y):=\Big( [(E')^{-1}] \prod_{i=n}^{1}\prod_{k'=k_i'}^{1}\sigma_{P_{i,k'}'}\Big) (Y)
\end{array}
\]
where RHS is similarly defined as before.

This completes a proof of existence of the map $\sigma$ for all $X$ and $Y$.

\medskip
{\bf 8.} Before we show {\bf 4)}, let us show its weaker (numerical) version:
\[
 \#(X)\  \ge \ \#(\sigma_Y(X)).
\]
where we mean by $\#(A)$ for $A\in \M(n,R)^\times$ the number of irreducible factors in the irreducible decomposition of $A\ (=\#$ of prime factors in $\det(A)$). If $X$ is irreducible (i.e.\ $\#(X)=1$), then $\sigma_Y(X)$ is either irreducible or $1_n$ so that the inequality holds. Then, using {\bf 4.}, 
 $\#(Y)-\#(\sigma_X(Y))=\#(X)-\#(\sigma_Y(X))\ge0$ 
 for an irreducible $X$. Then for an irreducible decomposition $X\!=\!X_1\cdots X_N$, we obtain 
 $\#(Y)\! \ge \! \#(\sigma_{X_1}(Y)) \! \ge \! \#(\sigma_{X_2}(\sigma_{X_1}(Y)))\! = \! \#(\sigma_{X_1X_2}(Y))\! \ge \! \cdots \! \ge \! \#(\sigma_{X_1\cdots X_N}(Y))\! =\! \#(\sigma_X(Y))$.  Again using {\bf 4.}, we obtain
  $\#(X)-\#(\sigma_Y(X))\! =\!  \#(Y)-\#(\sigma_X(Y)\ge0$.

\medskip
{\bf 9.} Let us show  {\bf 4)} by the double induction on
 $(u,v)\!=\!(\#(X),\#(Y))\in \Z_{\ge0}\! \times \! \Z_{\ge0}$.  The
 cases for $(u,0)$ or $(0,v)$ (i.e.\ the cases when $Y\! =\! 1_n$ or
 $X\! =\! 1_n$) are trivially true. The construction in
 {\bf 6.} shows that $\sigma_Y(X)$ for a $p$-irreducible element $X$ and
 a $q$-irreducible element $Y$ is a $p$-irreducible element, except for
 the case $X\! =\! Y$ and $\sigma_Y(X)\! =\! 1$, and its level is
 unchanged except $p\!=\!q$ and levels of $X$ and $Y$ coincides. This implies
 the statement 4) for the case
 $(u,v)\!=\!(1,1)$. Let us show that our construction of
 $\sigma$ using {\bf 1., 2.} and {\bf 3.} preserves the property {\bf
 4)}, respectively.

Let $X\in M_n$ and $Y\in \M(n,R)^\times$ such that  $\frac{\det(X)}{\det(\sigma_Y(X))} \ \in \ M $. 

{\bf 1.}  For any $Y'\in \M(n,R)^\times$, one has
\vspace{-0.1cm}
\[
\begin{array}{c}
\frac{\det(X)}{\det(\sigma_{YY'}(X))} =\frac{\det(\sigma_Y(X))}{\det(\sigma_{Y'}(\sigma_Y(X))} \frac{\det(X)}{\det(\sigma_Y(X))} \ \in \ M 
\end{array}
\vspace{-0.1cm}
\]

{\bf 2.} For any $E\in \GL(n,R)$, one has 
\vspace{-0.1cm}
\[
\begin{array}{c}
\frac{\det(X)}{\det(\sigma_{YE}(X))} =\frac{\det(X)}{\det(E^{-1}(\sigma_Y(X))} =\frac{\det(X)}{\det(\sigma_Y(X))} \ \in \ M 
\end{array}
\vspace{-0.1cm}
\]

{\bf 3.} If $X,Y\in M_n$
\vspace{-0.1cm}
\[
\begin{array}{c}
\frac{\det(Y)}{\det(\sigma_X(Y))} =\frac{1}{\det(u(X,Y))}\frac{\det(X)}{\det(\sigma_Y(X))}= \frac{\det(X)}{\det(\sigma_Y(X))} \ \in \ M 
\end{array}
\vspace{-0.1cm}
\]
This complete a proof of {\bf 4)}  and, hence, that of Theorem 3.   
\end{proof}

\noindent
{\bf Corollary.} {\it  For $X\!\in\! M_n$ and $Y\!\in\! \M(n,R)^\times$, \
$X\mid_lY \ \Leftrightarrow \ \sigma_Y(X)=1$.}

 \newpage
\noindent
{\bf Explicit formula of $u(X,Y)$ for irreducible $X$ and $Y$.}  

Let $X$ and $Y$ be irreducible normal forms $M(i\! :\! ({\bf d},p,{\bf 0}))$ and $M(j \!: \! ({\bf e},q,{\bf 0}))$, respectively.
Using {\bf 6.} of Proof of Theorem 3, we obtain:

\smallskip
{\bf Case i)}  \ If $i<j$, then 
$\begin{array}{lll}
u(X,Y)=M(j: (-[\frac{{\bf e}\!+\!e_i({\bf d}, p\!-\!1,{\bf 0})}{q}],1,{\bf 0})    \in \mathrm{U}(n,R)
\end{array},
$
where we denote by $[\frac{a}{q}]$ for $a\in R$ the unique element $r\in R$ such that $a-rq\in R(q)$.

\medskip
{\bf Case ii)}  \  If $i=j$ and $p\not=q$, then 
$\begin{array}{lll}
u(X,Y)\!&\!\!=\!\!&\! M(i:({\bf h},1,{\bf 0})) \in \mathrm{U}(n,R).
\end{array}
$
where ${\bf h}\in R^{i-1}$ is the unique solution of the equation: 
 ${\bf d}-{\bf e}=q{\bf f}-p{\bf g} +pq{\bf h}$ \quad for some unknown ${\bf f}\in R(p)^{i-1}$, ${\bf g}\in R(q)^{i-1}$ 
(this equation has a unique solution, since $p$ and $q$ are distinct primes).
 
\medskip
{\bf Case iii)}  \ Let $i=j$, $p=q$. If $X=Y$, then $\sigma(X,Y)=1_n$.  If $X\not=Y$, then  
$
\begin{array}{lll}
u(X,Y)\!&\!\!=\!\!&\! M(i\!:\!\!(g_1,\cdots,g_{k-1},d_k-e_k,\!\frac{d_{k+1}-e_{k+1}}{p},\cdots,\frac{d_{i-1}-e_{i-1}}{p},1,{\bf 0})) \in \mathrm{U}(n,R),
\end{array}
$
where $k\!:=\!\max\{1\! \le\! m\! <\! i \mid d_m\!-\!e_m\not\equiv 0 \bmod{p}\}$ and ${\bf g}\! \in\! R^{k-1}$ is the unique solution of the equation: 
$d_l\!-\!e_l\!+\!f_l(d_k\! -\! e_k)\! =\! g_l p\ \ (1\!\le\! l\! \le\! k\!-\!1)$ \quad for some unknown ${\bf f}\in R(p)^{k-\!1}$ (this equation has a unique solution 
since $d_k\!-\! e_k$ is prime to $p$).

\medskip
{\bf Case iv)} If  $i>j$, we reduce this case to  i) by $u(X,Y)=u(Y,X)^{-1}$.

  \bigskip
\noindent
 {\bf Note.}  As stated in Proof of Theorem 3., the steps {\bf 1, 2,
 3, 4} and {\bf 7.} are general properties valid for any
 cancellative monoids. Therefore, we formulate below the result. As we shall see in the proof of
 \S5 Theorem 4, Corollary below holds also.
 
 \medskip
 \noindent
 {\bf Theorem.}  {\it  Let $\mathcal{M}$ be a cancellative monoid. Let
 ${\bf M}$ be a subset of $\mathcal{M}$ which represents all left
 equivalence classes in $\mathcal{M}$ uniquely.  Set  $I_0:={\bf M}\cap\{\text{irreducible elements}\}$. 

 Suppose that there exists a map 
$ 
 \sigma  : I_0\times I_0  \rightarrow I_0, \   (Y,X) \mapsto \sigma_Y(X)
$ 
 such that for any $Z\in \mathcal{M}$ one has the equivalence: 
 \[
 \!\! (*) \qquad\qquad \qquad\qquad \qquad
X\  |_l \ YZ    \     \Longleftrightarrow  \  \sigma_Y(X)\  |_l\ Z  .
 \qquad\qquad \qquad\qquad\qquad
\]
Then the map $\sigma$ uniquely extends to a map $\sigma : \mathcal{M} \times {\bf M}  \to {\bf M}$  so that $(*)$ holds for any 
$X\in {\bf M}$ and $Y,Z\in  \mathcal{M}$.
The map $\sigma$ satisfies further the following {\bf 1)} - {\bf 3)}.

\medskip
\noindent
{\bf 1)}  The map $\sigma$ defines an opposite left action $\sigma_Y$ of $Y\in \mathcal{M}$ on the set ${\bf M}$ with the  fixed point $1_n$. That is,
\[
\sigma_{1_n}=id_{M_n}, \quad     \sigma_{Y_2Y_1}=\sigma_{Y_1}\circ \sigma_{Y_2} \quad \text{and}\quad  \sigma_Y(1_n)=1_n
\]
for any $Y,Y_1,Y_2\!\in\! \mathcal{M}$. 

\medskip
\noindent
{\bf 2)} The map $\sigma$ is compatible with the unit group action on $\mathcal{M}$ from right and that on ${\bf M}$ from left. That is,  for any $Y\in \mathcal{M}$ and an invertible $E\in \mathcal{M}$, we have 
\[
\sigma_{YE} = [E^{-1}]\sigma_Y 
\]
where $[E^{-1}]$ denotes the left action on the set ${\bf M}$ induced from the left action of $E^{-1}$ on the set of all left equivalence classes.

\medskip
\noindent
{\bf 3)}  There exists a map 
$ u:  {\bf M} \times {\bf M} \rightarrow 
\mathrm{U}(\mathcal{M})\!:=\!$ the unit group of  $\mathcal{M}$
preserving {\bf M} such that,  for $X,Y\!\in\! {\bf M}$, we have}
\[
X\sigma_X(Y) =Y\sigma_Y(X) \cdot u(X,Y).
\]
{\bf 4)} If there exists a degree map on $M$ (c.f.\ \S6) such that
$\deg(X)\ge\deg(\sigma_Y(X))$ for $X,Y\in I_0$, then the inequality holds for all
$X,Y\in {\bf M}$.

\smallskip
\noindent
{\bf Corollary.}\ {\it Under the setting of  Theorem, any finite  subset $J$ of $\mathcal{M}$ admits a unique least common multiple $\LCM(J) \in {\bf M}$ (up to the right unit factor).}


\section{\bf \large     Least common multiples}

As a consequence of the divisibility theory in the previous section, we describe the least common multiple for a given finite set  in $\M(n,R)^\times$ and its basic nature.

\medskip

\noindent
{\bf Definition.}  An element $Z\! \in\!  \M(n,R)^\times$ is called a {\it least common multiple} of  a set $J\!\subset\!  \M(n,R)^\times$, if 1) $X\!\mid_l \! Z$ $\forall X\!\in\! J$ and 2)  if $X\! \mid_l \! Z'$ $\forall X\! \in \!J$ for some $Z' \! \in\! \M(n,R)^\times$ then $Z\! \mid_l \!Z'$. 
By the definition, least common multiples of $J$ form either an empty set or a single left equivalence class.  In the latter case, we shall denote by $\LCM(J)$ the normal form of the class and call it {\it the left least common multiple}  of $J$.\!\!

\begin{theo}
Any finite $J\! \subset\!  \M(n,R)^\times$
has  the least common multiple $\LCM(J)$. 
\end{theo}

\begin{proof}  
We apply recursively Theorem 3 on the cardinality of $J$,  where the case $\#J=1$ is trivially true.
Let $\#J>1$ and put $J=J'\sqcup \{X\}$. By our induction hypothesis, there exists $\LCM(J')$.
Then,  $\LCM(J') \cdot \sigma_{\LCM(J')}(X)$ is a least common multiple of the set $J$, since 1)  it is divisible by any $X'\in J',$ and divisible by $X$ ($\Leftrightarrow \sigma_{Z \cdot \sigma_{Z}(X)}(X)=\sigma_{\sigma_Z(X)}(\sigma_Z(X))=1_n$), and 2) if an element $Z\in \M(n,R)^\times$ is divisible by the elements of $J'\sqcup \{X\}$ then $Z$ should be divisible by $\LCM(J')$ and by $X$, implying that $Z$ is divisible by $\LCM(J')\sigma_{\LCM(J')}(X)$.
\end{proof}

Combining the description $\LCM(X,Y)=[Y\sigma_Y(X)]$ with 4) the monotonicity of Theorem 3,
we obtain the following ``upper and lower bound'' of $\LCM(X,Y)$.

\begin{cor}
For any $X,Y\in \M(n,R)^\times$, we have
\[
\lcm(\det(X),\det(Y)) \ \mid \ \det(\LCM(X,Y)) \ \mid \  \det(X) \det(Y).
\]
\end{cor}
\begin{proof} The first division relation follows from:  $X|_l
  Y\sigma_Y(X)$ and $Y|_l X\sigma_X(Y)$, and the second division
  relation follows from
$\det(X)\det(Y)/\det(\LCM(X,Y))=\det(X)\det(Y)/\det(X\sigma_X(Y))=\det(Y)/\det(\sigma_X(Y))\in |R|\subset R$.
\end{proof}

\medskip
In the following Part I and II of the present section, we describe two
basic properties of $\LCM$, which will be used in \S6 to analyze  the skew growth function.  As a digression application of them, at the end of this section, we discuss about a generalization of the {\it fundamental element}, which was introduced for braid monoids by Garside and then for Artin monoids [B-S] and for other cases [S-I].
 
\medskip
  \noindent
  {\bf Part I.} We first study the behavior of $\LCM$ of elements whose diagonals are co-prime to each other at each level $i\!=\!1,\cdots,n$.  Then we get "multiplicativity" of the diagonals at each level. The result is used to show the weak Euler product decomposition of the skew growth function.

\begin{lemm}
Let $X,Y\! \in\! \M(n,R)^\times$ be with  $\diag([X])\! =\!
 (l_1,\cdots,l_n)$ and $\diag([Y])\!=\! (m_1,\! \cdots\! ,m_n)$. If
 $l_i$ and $m_i$ are relatively prime in $R$ for $1\!\le\! i\!\le\!n$
 (we shall say that $\diag([X])$ and $\diag([Y])$ are
 componentwisely co-prime), then 
$\diag(\LCM(X,Y))\!=\!(l_1m_1,\!\cdots\!,l_nm_n)$. In particular, we have:  $\det(X)\det(Y)\!=\!\det(\LCM(X,Y))$.
\end{lemm}
\begin{proof}
Let $m_i=\prod_{k=1}^{k_i}p_{i,k}$ be the prime decomposition of the diagonal entities of $Y$, and let $Y=(\prod_{i=1}^n\prod_{k=1}^{k_i} P_{i,k})E$, where $P_{i,k}$ is a $p_{i,k}$-irreducible normal form of level $i$, $E\in \GL(n,R)$ and the order of the product is lexicographic, be the irreducible decomposition of $Y$ given in \S7 Appendix Lemma 10. 
Then,  we have 
$Y \sigma_Y(X)\simeq \big(\prod_{i=1}^n\prod_{k=1}^{k_i} P_{i,k}\big)\big(\prod_{i=n}^{1}\prod_{k=k_i}^{1}\sigma_{P_{i,k}} (X)\big)$, where,  inside in each big parenthesis of RHS, product is lexicographic or anti-lexico-graphic order. Considering each action of $\sigma_{P_{i,k}}$ inductively, it is sufficient to prove the case when $Y $ is irreducible. Namely, we have only to prove the following special case.

\medskip
\noindent
{\bf Assertion.}  {\it Let $Y$ be a $p$-irreducible element of level $i$. Suppose the $i$th component $l_i$ of $\diag([X])$ is prime to $p$. Then, we have $\diag([X])=\diag(\sigma_Y(X))$.}
\begin{proof} The statement is equivalent to that $\diag([X\sigma_X(Y)])=\diag([Y\sigma_Y(X)])$ is equal to $\diag([X])$ except at the $i$th place, where the values is $pl_i$, and, then, it is equivalent to that $\sigma_X(Y)$ is $p$-irreducible of level $i$ if $Y$ is so and $X$ is lower triangular matrix whose $i$th diagonal element $l_i$ is prime to $p$.  This can be reduced again to the case when $X$ is irreducible, by considering the decomposition 
$X=(\prod_{i=1}^{n}\prod_{k'=1}^{k_i'}P_{i,k}' ) 
E'$ and considering the action of $\sigma_{P'_{i,k}}$ in the expression 
$X\sigma_X(Y)\simeq (\prod_{i=1}^{n}\prod_{k'=1}^{k_i'}P_{i,k}' ) \Big(\prod_{i=n}^{1}\prod_{k'=k_i'}^{1}\sigma_{P_{i,k}'}\Big) (Y)$ inductively. 
However,  in 6. of the proof of Theorem 3, it was shown that if $X$ and $Y$ are irreducibles, then $\sigma_X(Y)$ is an irreducible element such that i) $\det(Y)=\det(\sigma_X(Y))$ and ii) levels of $Y$ and $\sigma_X(Y)$ coincides each other except the case $\det(X)=\det(Y)$ and levels of $X$ and $Y$ coincides (when the jump of level occurs in Case ii) of the proof). 
\end{proof}

\vspace{-0.2cm}
This completes a proof of Lemma 6.
\end{proof}

\noindent
 {\bf Part II.} 
We next study the behavior of  $\LCM$ for a set of  $p$-irreducibles for
 a fixed prime $p\!\in\! R$. We will observe that the diagonals are no-longer multiplicative but  
the levels  may go down or disappear (jumping of the level),  that is,
 the data of levels of the input $J$ alone cannot determine the levels
 of the output $\LCM(J)$. 

\noindent
\begin{lemm}{\it Let $X\!=\!M(i:{\bf x})$ and $Y\!=\!M(i:{\bf y})$ be two distinct $p$-irreducible normal forms of the same level $i$. Set $k\!:=\!\max\{ 1\le l<i \mid x_l-y_l\not\equiv 0 \bmod{p}\}$. Then 
\[
\LCM(X,Y)=M(k,{\bf u})M(i,{\bf v}),
\]
where $M(k:{\bf u})$ and $M(i:{\bf v})$ are mutually commutative $p$-irreducible normal forms of level $k$ and $i$, where the row vectors ${\bf u}=(u_i)$ and ${\bf v}=(v_i)$ are given as follows. 
}
\[\begin{array}{rllll}
u_l  \equiv  (x_l-y_l) /(x_k-y_k)\quad \bmod{p} \text{ for } 1\!\le\! l\! <\!k, &\! u_k=p, &\! u_l=0 \text{ for }  k<l \le n\\
v_l  \equiv  (x_ly_k \! -\!  y_lx_k)/(x_k\! -\! y_k)\bmod{p} \text{ for } 1\le\! l \!<\! k,  &\! v_k=0, &\! v_l\!=\!x_l\!=\!y_l \text{ for  }k\!<\! l \!\le\! n. 
\end{array}
\]
Here we use the bijection $R/(p)\simeq R(p)$ for the reason given in {\bf Lemma 9}. 
\end{lemm}
\begin{proof} Recall the proof  of {\bf 6.} ii) of Theorem 3. Details are left to the reader. 
\end{proof}


For a set $J$ of $p$-irreducibles elements, $\LCM(J)$ can be
calculated by applying Lemma 6 and 7 successively.
Its final form is characterized in the following Lemma. 
 
\begin{lemm}
The following conditions i) - v) on $X\in M_n$ are equivalent.

\noindent
i)  There exists a set $J$ of $p$-irreducibles such that $X=\LCM(J)$.

\noindent
ii)  $X$ divides $p1_n$.

\noindent
iii)  $ X$ satisfies the following 1) and 2).

1) {\it Diagonal entries of $X$ are either equal to 1 or to $p$. }

2) {\it If the $i$th diagonal entry of $X$ is equal to 1, then the $(i,j)$-entry of $X$ for all 

\quad 
$j$ with $1\le j<i$ is equal to 0. If $j$th diagonal entry of $X$ is equal to $p$, then 

\quad 
the $(i,j)$-entry  of $X$ for all $i$ with $j<i\le n$ is equal to 0. }

\noindent
iv)  Let ${\bf x}_{ i}$ be the $i$th row-vectors of $X$ ($1\!\le\! i \!\le \! n$), and set $J(X):=\{M(M:{\bf x}_{i})\}_{i=1}^n$.
Then, 1) $J(X)$ consists of  mutually commutative $p$-irreducibles and possibly $1_n$, 

\quad \ \ 
2) 
$\begin{array}{ll}
X=\LCM(J(X))=\prod_{i\in \{1,\cdots,n\}} M(i,{\bf x}_{i}).
\end{array}
$ 

\noindent
v)  $X$ is a product of $p$-irreducible normal
 forms which are mutually commutative and mutually of different levels.
\end{lemm}
\begin{proof}
i) $\Rightarrow$ ii).  For a $p$-irreducible element $X$, $\det(X)\!=\!\varepsilon p$ ($\varepsilon\!\in\! \mathcal{E}$)  implies $X\mid_lp1_n$. Then, any  least common multiple of $p$-irreducible elements  should divides $p1_n$.

\smallskip
ii) $\Rightarrow$ iii).   1)  follows since $p\cdot X^{-1}$ is integral. 
The first half of 2) follows from the definition of a normal form and
 $R(1)=\{0\}$. 
Let $j$th diagonal entry of $X$ is equal to $p$. Put
 $\bar{i}:=\mathrm{min}\{j<i\le n\mid (i,j) \text{-entry of $X$ is not equal
 to } 0\}$. Due to 1) and the first half of 2), the $\bar{i}$th diagonal entry of $X$ is
 equal to $p$. Then the $(\bar{i},j)$-entry of $X^{-1}$ is of a form $c/p^2$
 for non-zero $c$, which contradicts to $pX^{-1}\in \M(n,R)$.

\smallskip
iii) $\Rightarrow$ iv). Since the diagonals of $X$ are either 1 or $p$, $J(X)$ consists of identity matrices $1_n$ and some $p$-irreducible normal forms of different levels
The commutativity of the elements of $J(X)$ follows from a general fact that 
{\it two normal forms $M(i:{\bf x})$ and $M(j:{\bf y})$ of levels $i$ and $j$, respectively, for $i<j$ are commutative if and only if $i$th entry of ${\bf y}$ is equal to 0}.  The commutativity implies $\LCM(J(X)) \big|_l \prod_{i=1}^{n} M(i:{\bf x}_{i})$. On the other hand, {\bf Lemma 6} implies that  $\det(\LCM(J(X)))$ is equal to  $p^k=\det(\prod_{i=1}^{n} M(i:{\bf x}_{ i}))=\det(X)$ where $k:=\#$ of $p$'s in the diagonal of $X$. Thus, the  equalities are shown. 

\smallskip
iv) $\Rightarrow$ v).  Clear.

\smallskip
v) $\Rightarrow$ i). If $X=X_1\cdot...\cdot X_m$ where $X_1,\ldots,X_m$
 are mutually commutative, then $\LCM(X_1,\ldots,X_m)\mid
 X$. If, further,  $X_1,\ldots,X_m$ are mutually of different level,
 then applying Lemma 6 inductively on $m$, we get $X=\LCM(X_1,\cdots,X_m)$.
\end{proof}

\noindent
{\it Note.}  The "commutativity" used in iv) and v) are not a property of the classes in $\M(n,R)^\times\!/\!\sim_l$ but a property of the matrices of normal forms  themselves.

\noindent
{\it Example.}  1.  A matrix like 
{\footnotesize $\begin{bmatrix} p\! &\! 0 \\  1\!& \!p\end{bmatrix}
=\begin{bmatrix} p\! &\! 0 \\ 0\!& \! 1\end{bmatrix}
\begin{bmatrix} 1\! &\! 0 \\ 1\!& \! p\end{bmatrix}$}, which violate the condition iii), 
\vspace{-0.1cm}
cannot be a least common multiple of some irreducible elements.

\noindent
2.\ If
{\footnotesize
$A\!\!:=\!\!\! \begin{bmatrix} p\!&\!0 \!&\! 0  \vspace{-0.1cm} \\  
0 \!&\! 1 \!&\! 0 \vspace{-0.1cm}\\ 
0 \!&\! 0 \!&\! 1\end{bmatrix}\!\!,
 B\!\!:=\!\! \! \begin{bmatrix} 1\!&\!0 \!&\! 0 \vspace{-0.1cm}\\  
 0 \!&\! 1 \!&\! 0 \vspace{-0.1cm}\\ 
 i \!&\! k \!&\! p\end{bmatrix}\!\!,
C\!\!:=\!\!\! \begin{bmatrix} 1\!&\!0\!&\! 0 \vspace{-0.1cm}\\ 
0 \!&\! 1 \!&\! 0 \vspace{-0.1cm}\\ 
j \!&\! k \!&\! p\end{bmatrix}$\!\!
} for $i\!\!\not=\!\!j,k\!\in\! R(p)$, then 
{\footnotesize 
$\LCM(B,C)\!=\!
\begin{bmatrix} p\!&\!0 \!&\! 0 \vspace{-0.1cm}\\ 
0\!&\! 1 \!&\! 0  \vspace{-0.1cm}\\ 
0 \!&\! k \!&\! p\end{bmatrix}$}

\noindent
is divisible\! by\! $A$. Then we have: 
{\footnotesize
$\LCM(A,B)\! =\! \LCM(B,C)\! =\! \LCM(C,A)\! =\! \LCM(A,B,C)$.  
}

\medskip
We give a useful criterion to be divisible by a $p$-irreducible element. 

\begin{lemm}
  {\it A $p$-irreducible element $X\!\in\! \M(n,R)^\times$ divides  an element $Y\in \M(n,R)^\times$ from the left, if and only if the mod $p$ reduction of $X$ divides that of $Y$ in $\M(n,R/(p))$ (here 
   "division relation in $\M(n,R/(p))$" is used in the sense given in the proof since 
 $\det(X \bmod{p})\equiv \det(X)\equiv 0 \bmod{p}$).}
\end{lemm}
\begin{proof}  
 Suppose there exists  $Z\!\in\! \M(n,R)$ such that
 $Y\equiv XZ\bmod{p}$. Then, there exists $W\in \M(n,R)$ such that  $XZ=Y+pW$.
 Using $X^*$ (=adjoint of $X$, i.e.\ an element $X^*\in \M(n,R)$ s.t.\ $XX^*=p1_n$), we get the expression $X(Z\!-\!X^*W)\!=\!Y$. Since $\det(Y)\!\not=\!0$, we get 
$\det(Z\!-\!X^*W)\!\not=\!0$, and hence $X\mid_l Y$ in
 $\M(n,R)^\times$. Conversely, if there exists $Z\in \M(n,R)^\times$
 with $Y=XZ$, then $Y\equiv XZ \bmod{p}$.
\end{proof}


\section{\bf \large     Growth function and skew-growth function} 
As an application of the divisibility theory of $\M(n,R)^\times$, we
study the growth and skew-growth function of the monoid $\M(n,R)^\times$.
When $R$ is residue finite, we give a direct proof of their Euler product formula.

We recall the definition of the skew-growth functions
([S4]). Let $M$ be a cancellative monoid with
the unit group $G$, which admits least common multiples.\footnote
{The skew-growth function [S4, \S4] is defined for arbitrary
cancellative monoid without assuming the existence of l.c.m.\ but using towers of common multiple sets. However, the existence of l.c.m.'s implies that  the height of the towers is 1 and we get the present simple formulation.
}
 
A map 
$\deg: M \longrightarrow \R_{\ge0}$ is called a {\it degree map} if it satisfies

\ \ i)  \ $\deg(X)=0$ if and only if $X\in G$,

\ ii)  \ $\deg(XY)=\deg(X)+\deg(Y)$ for all $X,Y\in M$,

iii) \ $\#(\{X\in M \mid \deg(X)\le r\}/G)<\infty$ for all $r\in \R_{>0}$.

\medskip
For a given degree map,  the growth function $P_{M,\deg}(t)$ and the skew growth function $N_{M,\deg}(t)$ are defined as formal Dirichlet series:
\[
\begin{array}{rll}

\vspace{0.1cm}
P_{M,\deg}(t) &:=& \sum_{[X]\in M/G}t^{\deg([X])}, 
\\
N_{M,\deg}(t) & :=&  \sum_{J: \text{ finite subset of }I_0}(-1)^{\#J}\sum  t^{\deg(\LCM(J))},
\end{array}
\vspace{0.1cm}
\]
where $I_0:=$ left equivalence classes of irreducible elements of $M$.
As formal  Dirichlet series, they satisfy the inversion formula ([S4,\S5]) 
\[
P_{M,\deg}(t)\ N_{M,\deg}(t) =1.
\]

 Let us call a domain $R$ to be {\it residue finite}, if  $\#(R/mR)\!<\!\infty$ for all $m\!\in\! R\!\setminus\! \{0\}$.  
\footnote
{ This condition is satisfied by 1) the principal order $R$ of an algebraic number field of class number  1, e.g.\  $R\!=\!\Z$, and 2) the coordinate ring of a smooth affine curve over a finite field.
}
The map $N:R\setminus \{0\}\to \Z_{\ge1}, m\mapsto \#(R/mR)$ is called the {\it absolute norm}. We define the degree map on $\M(n,R)^\times$ by the composition:
\[  
 \deg: = \log\circ N\circ \det \ : \ \ \M(n,R)^\times\longrightarrow\ \  \R_{\!\ge0}
\]
where $\log$ is the logarithmic function taking the branch: $\R_{\ge1}\to \R_{\ge0}$.  
The fact that this map satisfies the condition iii) follows from the following fact.

\medskip
\noindent
{\bf Fact.} {\it If $R$ is residue finite,\! then $\#(\{(m)\!\subset\! R\mid \#R/(m)\!\le\! r\})\! <\! \infty$ for any $r\! \in\! \R_{>\!0}$. }

\medskip
\noindent
{\it Proof} (Kurano).  Suppose the contrary. Then there are infinitely many distinct ideals $I_i\subset R$ such that all quotients $R/I_i$ are isomorphic to a fixed finite field of order, say, $f$. The natural map $R \to  \prod R\! /\! I_i$ is injective since for any $x \! \not= \! 0 \in R$, the ideal $(x)$ is contained in only finite many of  $I_i$'s. For any element $x\! \not=\! 0 \in R$, the image of $x^{f-1}$ in $R\!/\! I_i$ is equal to either $0$ or $1$ so that we have $x^{f-1}(x^{f-1}\! -\!1)\! =\! 0$. Since $x\not\! =\! 0$, the fact that $R$ is a domain implies  $x^{f\!-\! 1}\! - \!1\! =\! 0$. Since $R$ is infinite, taking mutually distinct  elements $x_1,\! \cdots\! ,x_f \!\in \! R\! \setminus\! \{0\}$, we get $x_i^{f\!-\!1}\!\!-\!1\!=\!0$ ($i\!=\! 1,\cdots\!,f$). This contradicts again to the fact that $R$ is a domain. \qquad \qquad  {$\Box$}

\medskip
In the rest of the present paper, we consider only the case when the
principal ideal domain $R$ is residue finite,  and consider only the growth and skew-growth functions associated with the degree map induced from the absolute norm.

\medskip
\noindent
{\bf Formulae.}  Let $R$ be residue finite.
Then, by a change $t=\exp(-s)$ of  variables from $t$ to $s$, the associated growth and skew-growth functions are absolutely convergent on some right half $s$-plane and are given as  analytic functions as follows.
\[
\!\!\!\!\!\!\!\!\!\!\!\!\!\!\!\!\!\!\!\!\!\!\!\!\!\!\!\!\!\!\!\!\!\!\!\!\!\!\!\!\!\!\!\!
{\bf 1)}\quad  
P_{\M(n,R)^\times,\deg}(\exp(-s))\ =\ \ \zeta_R(s)\ \zeta_R(s-1)\ \cdots\ \zeta_R(s-n+1)
\]
\[
\!\!\!\! \!\!  {\bf 2)}\ \   \ 
N_{\M(n,R)^\times,\deg}(\exp(-s))\ =\!\!\!\!\!\!\!\!\!\!\!\!
\prod_{p\in \{\text{primes of $R$}\}/\mathcal{E}}\!\!\!\!\!\!\!\!\!\!\!\!\!\!\!\!
{ (1\!-\!N(p)^{-\!s})(1\!-\!N(p)^{-\!s\!+\!1})\cdots (1\!-\!N(p)^{-\!s\!+\!n\!-\!1})}
\]
where 
$\zeta_R(s):=\sum_{a\in (R\setminus\{0\})/\mathcal{E}} N(a)^{-s}$
is the Dedekind zeta-function,
which is wellknown to be absolutely convergent on a region $\Re(s)\! >\! \exists\sigma_a$ and has the Euler product expression on 
$\prod_{p\in \{\text{primes of $R$}\}/\mathcal{E}}(1-N(p)^{-s})^{-1}$
on the same domain.

\begin{proof}  
1)  By the change of the variable, we rewrite the growth function
\[\begin{array}{ll}
\!\!\!\!\!\!\!\!\!\!\!\!\!\!\!\!\!\!\!\!\!\!\!\!\!\!\!\!\!\!
{\bf 1)'} \quad &\quad P_{\M(n,R)^\times,\deg}(\exp(-s))=\sum_{[X]\in \M(n,R)^\times/\GL(n,R)} N(\det(X)) ^{-s}.
\end{array}
\]
There are two proofs of the zeta function expression 1).
The first one is to regard the expression 1)' as a generalized Epstein zeta function $\zeta_n(1_n,s)$ for the quadratic form $X\in \M(n,R)\mapsto \det(^t\!X1_nX)\!=\!\det(X)^2$ (up to a factor  of 2), then the formula {\bf 1)} follows from classical results  (K.L.\ Siegel [Si], M.\ Koecher [Ko]).     $\Box$

Let us give an alternative elementary proof of {\bf 1)} using the normal forms $M_n$. 
Let $X\in M_n$ be a normal form (\S3) with $\diag(X)=(m_1,\cdots,m_n)$. Then by the definition of the degree map, we have 
\[
t^{\deg(X)}=t^{\log(N(m_1))+\cdots+\log(N(m_n))}=N(m_1)^{\log(t)}\cdots N(m_n)^{\log(t)}. 
\]

\noindent
Then, due to Lemma 2 and in view of $N(m)\!=\!\#(R(m))$ for $m\!\in\! |R|$, we have
\[
\begin{array}{llll}
 P_{\M(n,R)^\times,\deg}(t)\!\!& =&\! \sum_{X\in M_n} t^{\deg(X)} \\
 &=&\! \big( \sum_{m_1\in |R|}N(m_1)^{\log(t)} \big) \\
 &&\!\!\!\!\!   \times \big(\sum_{m_2\in |R|} \big(\sum_{d_{21}\in R({m_2})}N(m_2)^{\log(t)} \big)\big) \\
 && \quad \cdots \\
&&\!\!\!\!\!   \times \big( \sum_{m_n\in |R|} \big(\sum_{d_{n1}\in R(m_n)}\cdots\sum_{d_{n,n-1}\in R(m_n)}N(m_n)^{\log(t)}\big)\big)\\
\\
 &\!\!\!\!\!\!\!\!\!\!\!\!\!\!\!\!\!\!=&\!\!\!\!\!\!\!\!\!\!\!\!\!\!\! \!\!\!\!\!\sum_{m_1\!\in\! M} \!N(m_1\!)^{\log(t)} \!\sum_{m_2\!\in\! M} \!N(m_2\!)^{\log(t)\!+\!1} \!\cdots\!
 \sum_{m_n\!\in\! M} \!N(m_n\!)^{\log(t)\!+\!n\!-\!1},
\end{array}
\]
where we use a fact $\#(R(m))\!=\!N(m)$ for $m\!\in\! M$. 
Recalling the fact $M\simeq (R\!\setminus\!\{0\})/\mathcal{E}$, we have
 $\sum_{m\in |R|}N(m)^{\log(t)}=\zeta_R(-\log(t))$ and, hence, the formula 1).

\medskip
2) There are two proofs of the Euler product formula {\bf 2)}.  

The first proof is, as explained in Introduction, to rewrite the formula
 1)  by the well-known Euler product formula of the Dedekind
 zeta-function, and to apply the inversion formula ([S4]).  

In the following, we present another proof of {\bf 2)}, which uses the structure of common multiples studied in \S5 but does not use the inversion formula.  

Let us, first, show a partial Euler product expansion of skew-growth
 functions in the sense of the following  formula 4) in next Assertion.
 
\begin{asse} For any finite subset $J\subset I_0$, one has an addition formula:
 \[
  \begin{array}{ll}
  \!\!\!\! \!\!  \!\!\!\! \!\! \!\!\!\! \!\! \!\!\!\! \! 
  {\bf 3)} \qquad \qquad & \ \   \ 
\deg\big(\LCM(J)\big) =\sum_{p: \text{\rm  primes of } R} \deg\big(\LCM(J\cap I_{0,p})\big) .
 \end{array}
\]
where $I_{0,p}\! :=$\{left equivalence classes of all $p$-irreducible elements\}.
 Then we get
 \[
 \begin{array}{ll}
  \!\!\!\! \!\! 
   {\bf 4)} \qquad \qquad & \ \   \ 
 N_{M,\deg}(t)=\prod_{p: \text{\rm  primes of } R}\left(\sum_{J\subset I_{0,p}} (-1)^{\#(J)} \ t^{\deg(\LCM(J))}\right) .
\end{array}
 \]
 \end{asse}
 \begin{proof}
 Obviously,  $\LCM(J)\! =\!\LCM\big(\{\LCM(J\cap I_{0,p})\mid p \text{:
  primes  of }R\}\big)$, where, except for finite primes $p$, the
  intersection $J\cap I_{0,p}$ is empty and $\LCM(J\cap I_{0,p})=1_n$. 
  Then, the diagonal part $\diag(\LCM(J))$ of LHS is equal to that of
  RHS, which is given by the componentwise product of diagonal part of $\diag(\LCM(J\cap I_{0,p}))$ for primes
  $p$ of $R$, since they are mutually componentwisely co-prime and we can apply Lemma 6.
Thus, we have  $\det(\LCM(J))=\prod_{p:primes}\det(\LCM(J\cap I_{0,p}))$ and 
 {\bf 3)}.  Then applying  {\bf 3)} to the
  definition of $N_{M,\deg}(t)$, we obtain {\bf 4)}. 
 \end{proof}
 
To get the finer decomposition {\bf 2)}, it  remains  to show the decomposition 
 \[
 \begin{array}{ll}
  \!\!\!\! \!\! \!\!\!\! \!\! \!\!\!\! \!\! \! \!\! \! \!\! 
{\bf 5)} \qquad \qquad & \ \ 
 \sum_{J\subset I_{0,p}} (-1)^{\#(J)} \ t^{\deg(\LCM(J))}=\prod_{i=1}^n(1\!-\!N(p)^{-\!s\!+\!i \!-\! 1})
\end{array}
 \]
for each prime $p$ of $R$, where $-s\!=\!\log(t)$.  

Set $I_{0,p}=\sqcup_{i=1}^n I_{0,p}^{(i)}$ where $I_{0,p}^{(i)}:=\{X\in I_{0,p}\mid X \text{ is of level } i\}$  and $J^{(i)}\!:=\!J\!\cap\! I_{0,p}^{(i)}$ for  $J\subset I_{0,p}$.   We decompose the summation index set of {\bf 5)} as $2^{I_{0,p}}\!\!\!=\!A\!\sqcup \! B$, where $A\!:=\! \{J\! \subset \! I_{0,p}\mid \#(\!J^{(i)}\!)\! \le\!1\  (1\!\le\! \forall i\! \le\!  n)\}$
 and $B\!:=\!2^{I_{0,p}}\!\setminus A$.\footnote{The decomposition $2^{I_{0,p}}=A\cup B$ is ``suggested'' by Lemma 8, where $\LCM(J)$ for any $J\in 2^{I_{0,p}}$ is given by another $LCM(J')$
 for $J\in A$.}
 Then the proof of the formula {\bf 5)} is achieved if we show the following two formulae.
\[
 \begin{array}{lrrl}
\!\! \!\! \!\! \!\! \!\! \!\! \!\! \!\! \!\! \! \!\! \! \!\! \! \!\! 
{\bf 6)} \qquad\qquad &
 \sum_{J\in A} (-1)^{\#(J)} \ t^{\deg(\LCM(J))}\!\!&\!\!=\!\!&\prod_{i=1}^n(1\!-\!N(p)^{-\!s\!+\!i \!-\! 1}),\\
\!\! \!\! \!\! \!\! \!\! \!\! \!\! \!\! \!\! \! \!\! \! \!\! \! \!\!  
{\bf 7)} \qquad\qquad &
 \sum_{J\in B} (-1)^{\#(J)} \ t^{\deg(\LCM(J))}\!\!&\!\!=\!\!& 0.
 \end{array}
 \]

\noindent
{\it Proof of} {\bf 6)}.  Since for any $J\!\in\! A$, elements in $J$  consist of $p$-irreducibles elements of different levels, we can apply Lemma 6 repeatedly. Then,  we get 
 \[
  \begin{array}{ll}
\quad \! {\bf 3')} \quad  &      
\deg\big(\LCM(J)\big) =\deg\big(p^{\#(J)}\big)=\log(N(p)^{\#(J)})=\sum_{i=1}^n\log(N(p)^{\#(J^{(i)})})
. 
 \end{array}
\]
so that, similarly to the formula {\bf 4)}, we get
 \[
 \begin{array}{lllll}
  \!\!\!\! \!\! \!\!\!\!  
   {\bf 4')} \qquad\quad    &
&\sum_{J\in A} (-1)^{\#(J)} t^{\deg(\LCM(J))} \\
&= & \prod_{i=1}^{n}(1 - \sum_{X\in I_{0,p}^{(i)}} N(p)^{-s})  & = & \prod_{i=1}^{n}(1 -  N(p)^{- s + i - 1}),
\end{array}
 \]
 where, in the last step, we use the fact $\#(I_{0,p}^{(i)}) =\#(R(p))^{i-1})=N(p)^{i-1}$. 

\smallskip
\noindent
{\it Proof of} {\bf 7)}.  It is sufficient to show an existence of an involution map $\tau:  B \to B$ satisfying the following conditions: 

\smallskip
 i) $\LCM(J)\!=\LCM(\tau(J))$ \ \ and \ \ ii) $\#(J)+\#(\tau(J))\equiv 1\bmod{2}$ for all $J\in B$, 
 
 \noindent
 since then 
\[ 2\sum_{J\in B}\!\!\! (\!-\!1)^{\#(J)} t^{\deg(\LCM(J))}\!\!=\!\!\sum_{J\in B}\!\!\big ((\!-\!1)^{\#(J)} t^{\deg(\LCM(J))}\!+\!(\!-\!1)^{\#(\tau(J))}  t^{\deg(\LCM(\tau(J)))}\big)\!\!=\!\!0.
\]
If $n=1$, then $B=\emptyset$. Assume $n\ge2$.  We construct the involution $\tau$ by a use of the jump (\S5 Part II).  For $J\in B$, set 
$m\!:=\!\max\{ 2\!\le\!m\!\le\! n \mid \#(J^{(m)})\!\ge2\!\}$. According to \S5 Part II. Lemma 7 and 8, 
we have a decomposition 
$\LCM(J^{(m)})\! =\! \prod_{i\in \{1,\cdots,r\}} M(k_i:{\bf x}_{ k_i})$, where 
$M(k_i:{\bf x}_{k_i})$ ($1\le i\le r$) are mutually commutative $p$-irreducible normal forms of level $k_i$. By the jump phenomenon, we know that $r\ge 2$ and $1\!\le\! k_1<\cdots<k_r=m$, and, in particular,  $k_1<m$.
Then we define
\[
\tau(J):=\begin{cases}
J\sqcup\{M(k_1:{\bf x}_{k_1})\}   &\text{ if } M(k_1:{\bf x}_{k_1})\not\in J\\
J\setminus\{M(k_1:{\bf x}_{k_1})\}    & \text{ if } M(k_1:{\bf x}_{k_1})\in J.
\end{cases}
 \]
 It is clear that $\tau$ defines an involution of $B$. Let us show the
 properties i) and ii) of $\tau$, where ii) is apparent by
 definition. To see ii), let us decompose $J=J^{(m)}\cup J'$ and
 $\tau(J)=J^{(m)}\cup J''$ with
 $J':=J\setminus J^{(m)}$ and $J'':=\tau(J)\setminus J^{(m)}$ for $J\in B$. Then,
\[\begin{array}{rll}
 \LCM(J)\! &\!=\LCM(\{\LCM(J^{(m)}),J'\})=\LCM(\{M(k_i:{\bf x_{k_i}})\
  (i=1,\cdots,r),J'\})\\
 \LCM(\tau(J))\! &\!=\LCM(\{\LCM(J^{(m)}),J''\})=\LCM(\{M(k_i:{\bf
  x_{k_i}})\ (i=1,\cdots,r),J''\})
\end{array}
\]
where both RHS coincides to each other, since only difference between $J'$ and
 $J''$ is that one of them contains the element $M(k_1:{\bf x_{k_i}})$
 and the other does not.

\medskip
This completes the proof of the formula {\bf 7)} and, hence, of  the formula {\bf 2)}.
\end{proof}


\section{\bf \large  Appendix. Irreducible decomposition}   

Any element of $\M(n,R)^\times$ is decomposable into a product of
irreducible elements. However the decomposition is not unique and has a
big variety. In the
present Appendix, we give a decomposition, which is
used in the proof 7) of Theorem 3. 

\begin{lemm} Let $X\! \in\! \M(n,R)^\times$ and let
 $\diag([X])\! =\! (m_1,\cdots,m_n)$ be the diagonal of its normal form.
Let us fix an ordered irreducible decomposition $m_i=\prod_{k=1}^{k_i}
 p_{i,k}$ for each $m_i$ $(i=1,\!\cdots\!,n)$. Then, there exist a unique
 $p_{i,k}$-irreducible normal form $P_{i,k}$ of level $i$ for $1\le i\le
 n$ and $1\le k\le k_i$, and a unit element $E\in
 \GL(n,R)$ such that
\vspace{-0.1cm}
\[
\vspace{-0.1cm}
X  = \Big( \prod_{i=1}^n\prod_{k=1}^{k_i} P_{i,k}\Big) \ E .
\] 
Here the product order  is the lexicographic 
 order of the running index $i$ and $k$.
\end{lemm}
\begin{proof} 
We fix a notation: for $1\le i\le n$ and $p\in |R|\subset R$, we set
\[
 M(i:p):=\text{ the set of normal forms of level $i$ with diagonal $p$}.
\]

According to the ordered product of $m_i$,  we consider the ordered product set
\vspace{-0.1cm}
\[
P:=\prod_{k=1}^{k_1} M(1: p_{1,k})\prod_{k=1}^{k_2} M(2: p_{2,k}) \cdots \prod_{k=1}^{k_n} M(n: p_{n,k})
\]
and consider a product map $\pi : P \rightarrow \M(n,R)^\times$
 according to the order.
Then, in view
 of Lemma 2 and the identification $[X]_l=X\cdot \GL(n,R)$, Lemma 10
 is equivalent to the say that {\it the map $\pi$ induces a bijection to a subset
 of $\M(n,R)^\times$ 
 which is in one to one correspondence (by the left equivalence) with the set} 
\[
 M_n(m_1,\cdots,m_n):=\{X\in M_n\mid \diag(X)=(m_1,\cdots,m_n)\}.
\]
We prove the statement by induction on the number
 of factors in $P$. So, consider the product set $\tilde P$ by
 forgetting the last factor from $P$, i.e.\
$P=\tilde{P}\times M(n:p_{n,k_n})$,
 and put $m_n=\tilde{m}_np_{n,k_n}$.  The induction hypothesis says that the image
 $\pi(\tilde{P})$ is bijective to a set which is in one to one
 correspondence with $M_n(m_1,\cdots,m_{n-1},\tilde{m}_n)$.
Precisely, this means that $\pi(\tilde{P})$ consists of elements of the form
 $X=(\prod_{i=1}^{n-1}M(i:({\bf x}_i,m_i,{\bf 0})))M(n:({\bf y},\tilde{m}_n))$ 
for  ${\bf x}_i\in R^{i-1}$ and ${\bf y}\in
 R^{n-1}$ such that the set of their left-equivalence class $\{[X]_l\mid X\in
 \pi(\tilde{P})\}$ is bijective to $M_n(m_1,\cdots,m_{n-1},\tilde{m})$.
That is, ${\bf x}_i$ runs over the set which is (by taking $\bmod\ m_i$)
 bijective to $R(m_i)^{i-1}$ for $1\le i\le n-1$ and 
${\bf y}$ runs over the set $S\subset R^{n-1}$ which is (by taking $\bmod\ \tilde{m}_n$)
 bijective to $R(\tilde{m}_n)^{n-1}$.
We have to show that the set $\pi(P)=\pi(\tilde{P})M(n:p_{n,k_n})=\{XZ\mid X\in
 \pi(\tilde{P}), Z\in M(n:p_{n,k_n})\}\subset \M(n,R)^\times$ is
 (by the left-equivalence) bijective to
 $M_n(m_1,\cdots,m_n)$. Actually, $Z\in M(n:p_{n,k_n})$ is of the form
 $M(n:({\bf z},p_{n,k_n}))$ where ${\bf z}$ is running over the set
 $R(p_{n,k_n})^{n-1}$. Then the product $XZ$ is of the form
\vspace{-0.2cm}
\[
 XZ=\big(\prod_{i=1}^{n-1}M(i:({\bf x}_i,m_i,{\bf 0}))\big)M(n:({\bf
 y}+\tilde{m}_n{\bf z},m_n)). 
\vspace{-0.1cm}
\]
Then, we need to show that the set $T:=\{{\bf y}+ \tilde{m}_n{\bf z} \mid 
{\bf y}\in S, {\bf z}\in R(p_{n,k_n})^{n-1}\}$ is (by taking $\bmod\
 m_n$) bijective to $R(m_n)^{n-1}$. But, this is trivial, since, for each ${\bf w}\in
 R(\tilde{m}_n)^{n-1}$, there exist a unique element ${\bf y}\in S$ and ${\bf
 v}_{{\bf w}}\in R^{n-1}$ such that ${\bf y}={\bf w}+\tilde{m}_n {\bf v}_{ {\bf w}}$. 
Then, changing the index from ${\bf y}$ to ${\bf w}$, we have an expression 
\[
 T=\{ {\bf w}+\tilde{m}_n({\bf v}_{ {\bf w}}+{\bf z})
 \mid {\bf w}\in R(m_n)^{n-1},  {\bf z}\in R(p_{n,k_n})^{n-1}\}.
\]
In view of the exact sequence $0\!\to\! R/(m_n)\!\to\! R/(\tilde{m}_n)\!\to\! R/(p_{n,k_n})\!\to\! 0$, we get $T \bmod m_n=
\{ {\bf w}\!+\!\tilde{m}_n{\bf z}
 \mid {\bf w}\!\in\! R(m_n)^{n-1},  {\bf z}\!\in\! R(p_{n,k_n})^{n-1}\} 
 \bmod m_n = R(m_n)^{n-1} \bmod m_n$. 

This completes a proof of Lemma 10.
\end{proof}

\noindent
{\bf Remark 1.} If $X\in M_n$ in the case $R=\Z$ with \S3 {\bf Eg.},
then $E=1_n$.

\noindent
{\bf 2.} The irreducible factors $P_{i,k}$ depends strongly on the order
of the product.

Here are some examples of different irreducible normal forms decomposition.
\[
 \begin{bmatrix}
1\!&\! 0\\ 
1\!&\! 6
\end{bmatrix}
\! = \!
\begin{bmatrix}
1\!&\!0\\ 
 1\!&\!3
\end{bmatrix}
\begin{bmatrix}
1\!&\!0\\ 
0\!&\! 2
\end{bmatrix}
\!=\!
\begin{bmatrix}
1\!&\!0\\ 
1\!&\!2
\end{bmatrix}
\begin{bmatrix}
1\!&\!0\\ 
 0\!&\! 3
\end{bmatrix}
,\
 \begin{bmatrix}
 1\!&\! 0\\ 4\!&\! 6
\end{bmatrix}
\!=\!
\begin{bmatrix}
1\!&\!0\\ 
1\!&\!3
\end{bmatrix}
\begin{bmatrix}
1\!&\!0\\ 
 1\!&\! 2
\end{bmatrix}
\!=\!
\begin{bmatrix}
1\!&\!0\\ 
0\!&\!2
\end{bmatrix}
\begin{bmatrix}
1\!&\!0\\ 
 2\!&\! 3
\end{bmatrix}
\]
{\bf 3.} The product map: $M_n(m_1,\cdots,m_n)\times
 M_n(l_1,\cdots,l_n) \to M_n(m_1l_1,\cdots,m_nl_n)$, where the target
 set is considered as a
 subset of the quotient set $\M(n,R)/\GL(n,R)$, for general
 $m_i,l_i\!\!\in \!\!
 M$ is neither injective nor surjective.
Eg.\! For $(m_1,m_2)\!\!=\!\!(a,2)$ and $(l_1,l_2)\!\!=\!\!(2,b)$, we have two decompositions 
$
 \begin{bmatrix}
2a\!\!&\!\! 0\\ 
2c\!\! &\!\! 2b
\end{bmatrix}
\!\!=\! \!
\begin{bmatrix}
a\!&\!0\\ 
0\!&\!2
\end{bmatrix}
\!\!
\begin{bmatrix}
2\!&\!0\\ 
c\!&\! b
\end{bmatrix}
\!\! = \!\!
\begin{bmatrix}
a\!&\!0\\ 
 1\!&\!2
\end{bmatrix}
\!\!
\begin{bmatrix}
2\!\!&\!\!0\\ 
 c\!-\!\!1\!\!&\!\! b
\end{bmatrix}
$.
However $
 \begin{bmatrix}
2a\!\!&\!\! 0\\ 
2c\!+\!1\!\! &\!\! 2b
\end{bmatrix}
$ has no decomposition
$
\begin{bmatrix}
a\!&\!0\\ 
*\!&\!2
\end{bmatrix}
\begin{bmatrix}
2\!&\!0\\ 
*\!&\! b
\end{bmatrix}
E \equiv 0 \bmod 2$.
{\it A sufficient condition for the map to be
 bijective is that there exists $1\!\le\! i_0\!\le\! n$ such that $m_i\!=\!1$ for
 $i\! >\! i_0$ and $l_i\! =\! 1$ for $i\! <\! i_0$} (the proof is the same as Lemma
 10). 
The examples show that the lexicographic
 order using levels in Lemma 10 is necessary.
 Once one violate the the ordering of levels, then one loose the uniqueness or existence.


%

\begin{flushright}
\begin{small}
Iinstitute for Physics and Mathematics of Universe, \\
University of Tokyo, 
Kashiwa, Chiba, 277-8568 JAPAN
\end{small}
\end{flushright}

\begin{flushright}
\begin{small}
e-mail address :  kyoji.saito@ipmu.ac.jp
\end{small}
\end{flushright}

\end{document}